\documentclass[a4paper,12pt,reqno]{amsart}
\usepackage[T1]{fontenc}
\usepackage[utf8]{inputenc}
\usepackage{amssymb,dsfont,mathrsfs}
\usepackage[margin=1in]{geometry}
\usepackage{enumitem}
\usepackage[mathscr]{eucal}
\usepackage{graphicx,xcolor}
\usepackage[bookmarksdepth=2]{hyperref}

\numberwithin{equation}{section}

\newtheorem{theorem}{Theorem}[section]
\newtheorem{lemma}[theorem]{Lemma}
\newtheorem{corollary}[theorem]{Corollary}
\newtheorem{proposition}[theorem]{Proposition}

\theoremstyle{definition}

\newtheorem{remark}[theorem]{Remark}
\newtheorem{remarks}[theorem]{Remarks}

\DeclareMathOperator{\re}{Re}
\DeclareMathOperator{\im}{Im}

\DeclareMathOperator{\sign}{sign}

\DeclareMathOperator{\ind}{\mathds{1}}

\newcommand{\cm}{\mathscr{CM}}
\newcommand{\am}{\mathscr{AM}}
\newcommand{\amcm}{\am\text{-}\cm}
\newcommand{\laplace}{\mathscr{L}}

\newcommand{\C}{\mathds{C}}
\newcommand{\R}{\mathds{R}}
\newcommand{\Z}{\mathds{Z}}

\newcommand{\ph}{\varphi}

\renewcommand{\le}{\leqslant}
\renewcommand{\ge}{\geqslant}

\usepackage{ifthen}
\newcommand{\formula}[2][nolabel]%
{%
 \ifthenelse{\equal{#1}{nolabel}}%
 {\begin{align*} #2 \end{align*}}%
 {%
  \ifthenelse{\equal{#1}{}}%
  {\begin{align} #2 \end{align}}%
  {\begin{align} \label{#1} #2 \end{align}}%
 }%
}

\begin{document}

\title[Characterisation of the class of bell-shaped functions]{Characterisation of the class \\ of bell-shaped functions}
\author{Mateusz Kwaśnicki, Thomas Simon}
\thanks{Work supported by the Polish National Science Centre (NCN) grant no.\@ 2015/19/B/ST1/01457}
\address{Mateusz Kwaśnicki \\ Faculty of Pure and Applied Mathematics \\ Wrocław University of Science and Technology \\ ul. Wybrzeże Wyspiańskiego 27 \\ 50-370 Wrocław, Poland}
\email{mateusz.kwasnicki@pwr.edu.pl}
\address{Thomas Simon \\ Laboratoire Paul Painlevé \\ Université de Lille \\ Cité Scientifique \\ F-59655 Villeneuve d'Ascq Cedex \\ France}
\email{thomas.simon@univ-lille.fr}
%\date{\today}
\subjclass[2010]{Bell-shape, Pólya frequency function, completely monotone function, absolutely monotone function, Stieltjes function, generalised gamma convolution, Port's inversion formula}

\begin{abstract}
A non-negative function $f$ is said to be \emph{bell-shaped} if $f$ tends to zero at $\pm \infty$ and the $n$-th derivative of $f$ changes its sign $n$ times for every $n = 0, 1, 2, \ldots$\, We provide a complete characterisation of the class of bell-shaped functions: we prove that every bell-shaped function is a convolution of a \emph{Pólya frequency function} and an \emph{absolutely monotone-then-completely monotone} function. An equivalent condition in terms of the holomorphic extension of the Fourier transform is also given. As a corollary, various properties of bell-shaped functions follow. In particular, we prove that bell-shaped probability distributions are infinitely divisible, and that the zeroes of the $n$-th derivative of a bell-shaped function grow at a linear rate as $n \to \infty$.
\end{abstract}

\maketitle

%
%                            ---------- o ----------
%

\section{Introduction}
\label{sec:intro}

\subsection{Bell-shaped functions}
\label{sec:intro:bs1}

A smooth function $f : \R \to [0, \infty)$ is said to be \emph{bell-shaped} if it converges to zero at $\pm\infty$ and if for every $n = 0, 1, 2, \ldots$ the $n$-th derivative $f^{(n)}$ of $f$ changes its sign exactly $n$ times. The notion of a bell-shaped function originated in the theory of games (see Section~6.11.C in~\cite{karlin}). The prototypical example is the Gauss--Weierstrass kernel $G_t(x) = (4 \pi t)^{-1/2} e^{-x^2 / (4 t)}$. Other examples include density functions of the Cauchy distribution: $\pi^{-1} (1 + x^2)^{-1}$, the Lévy distribution: $\pi^{-1/2} x^{-3/2} e^{-1/x} \ind_{(0, \infty)}(x)$, and the hyperbolic secant distribution: $(\pi \cosh x)^{-1}$. More generally, as can be easily verified by an explicit calculation, for every $p > 0$, the functions $(1 + x^2)^{-p}$, $x^{-p} e^{-1/x} \ind_{(0, \infty)}(x)$ and $(\cosh h)^{-p}$ are bell-shaped. It is much less obvious to prove that $(1 + x^2)^{-1} (p^2 + x^2)^{-1}$ is bell-shaped whenever $p > 0$; however, $(1 + x^2)^{-1} (9 + x^2)^{-1} (16 + x^2)^{-1}$ is not bell-shaped; see Section~6.5 in~\cite{bell}.

I.J.~Schoenberg conjectured that there are no compactly supported bell-shaped functions. This was proved to be true by I.I.~Hirschman in~\cite{hirschman}. The results of I.J.~Schoenberg on total positivity imply that P\'olya frequency functions, if smooth, are bell-shaped; see~\cite{hw}. W.~Gawronski claimed in~\cite{gawronski} that density functions of stable distributions are bell-shaped; however, his argument contained an error, which rendered the proof incorrect unless the index of stability was in the set $\{2, 1, \tfrac{1}{2}, \tfrac{1}{3}, \tfrac{1}{4}, \ldots\}$. A rigorous proof for density functions of all stable distributions concentrated on a half-line was given by the second author in~\cite{simon}, and an extension to distributuions of hitting times of (generalised) diffusions was soon given in~\cite{js} by W.~Jedidi and the second author.

A broad class of bell-shaped functions was desribed recently by the first author in~\cite{bell}. Following~\cite{bell}, we use the term \emph{strictly bell-shaped} for the notion of bell-shape introduced above, and we say that $f$ is a \emph{weakly bell-shaped function} if $f$ is a non-negative measure such that the convolution of $f$ with every Gaussian $G_t$ is strictly bell-shaped. Every strictly bell-shaped function is weakly bell-shaped, and every smooth weakly bell-shaped function is bell-shaped; see Corollary~4.4 in~\cite{bell}. The following is the main result of~\cite{bell}.

\begin{figure}
\centering
\includegraphics[width=0.8\textwidth]{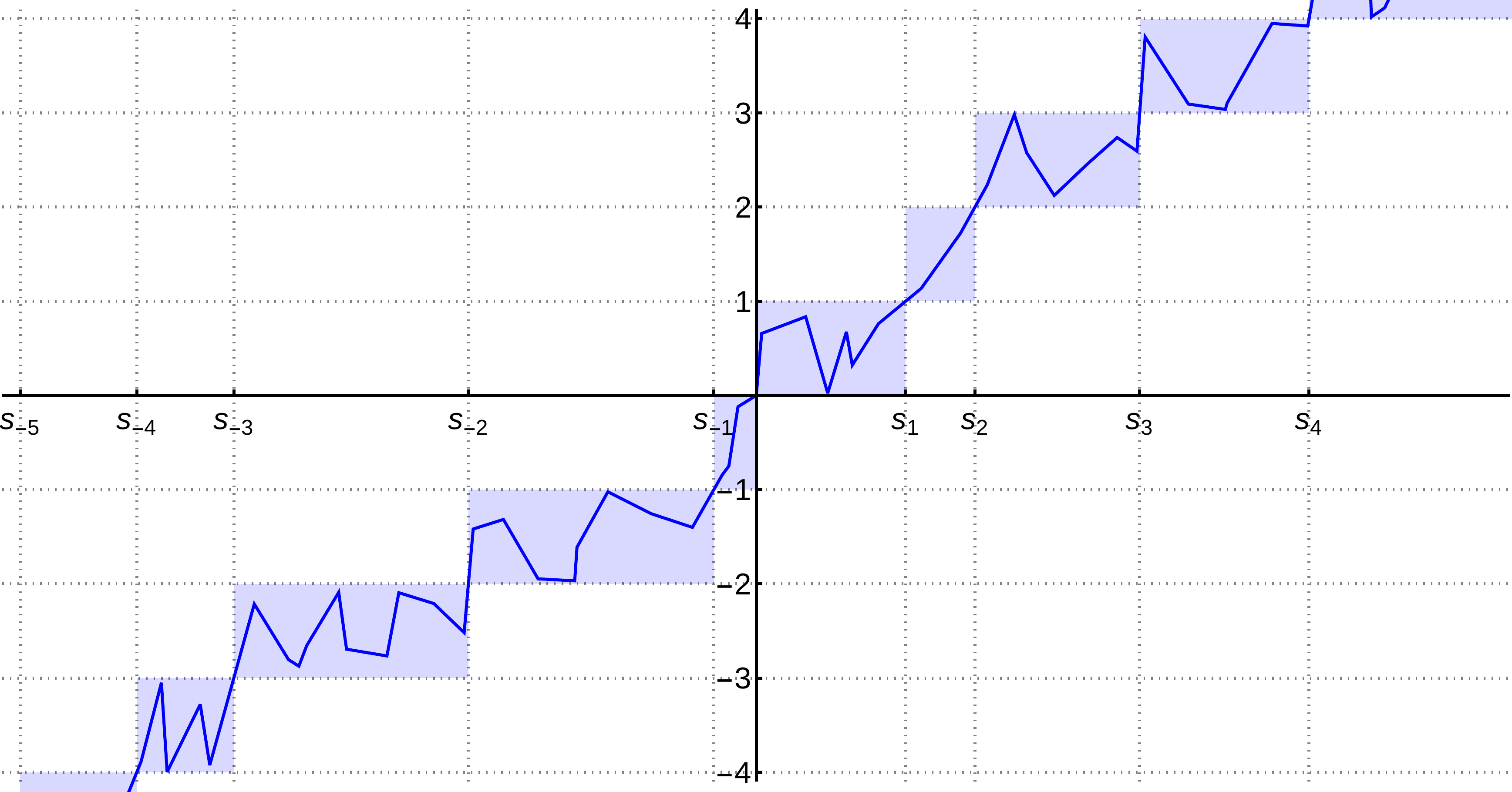}
\caption{Illustration for the level crossing condition~\ref{thm:bell:a} in Theorem~\ref{thm:bell}. The graph of the function $\ph$ is to be contained in the union of rectangles $[s_k, s_{k+1}] \times [k, k+1]$, $k \in \Z$.}
\label{fig:phi}
\end{figure}

\begin{theorem}[{Theorem~1.1 in~\cite{bell}}]
\label{thm:bell}
Suppose that $a \ge 0$, $b \in \R$, $c \in \R$, and that $\ph : \R \to \R$ is a Borel function with the following properties:
\begin{enumerate}[label=\textnormal{(\alph*)}]
\item\label{thm:bell:a} the level-crossing condition: for every $k \in \Z$ the function $\ph(s) - k$ changes its sign at most once, and for $k = 0$ this change takes place at $s = 0$: we have $\ph(s) \ge 0$ for $s > 0$ and $\ph(s) \le 0$ for $s < 0$ (see Figure~\ref{fig:phi});
\item\label{thm:bell:b} integrability condition: we have
\formula{
 \biggl( \int_{-\infty}^{-1} + \int_1^\infty \biggr) \frac{|\ph(s)|}{|s|^3} \, ds < \infty .
}
\end{enumerate}
For $\xi \in \R \setminus \{0\}$ define
\formula[eq:bell]{
 \Phi(\xi) & = \exp\biggl(-a \xi^2 - i b \xi + c + \int_{-\infty}^\infty \biggl( \frac{1}{i \xi + s} - \biggl(\frac{1}{s} - \frac{i \xi}{s^2} \biggr) \ind_{\R \setminus (-1, 1)}(s) \biggr) \ph(s) ds \biggr) ,
}
and assume in addition the following property:
\begin{enumerate}[resume*]
\item\label{thm:bell:c} regularity condition: we have
\formula{
 \int_{-1}^1 \re \Phi(\xi) d\xi & < \infty &\qquad& \text{and} &\qquad \lim_{\xi \to 0} (\xi \im \Phi(\xi)) & = 0 .
}
\end{enumerate}
Then there is a weakly bell-shaped function $f$ such that $\Phi$ is the Fourier transform of $f$:
\formula{
 \laplace f(i \xi) & := \int_{-\infty}^\infty e^{-i \xi x} f(x) dx = \Phi(\xi)
}
for every $\xi \in \R \setminus \{0\}$ (with the integral understood as an improper integral if $f$ is not integrable).
\end{theorem}

\begin{remarks}
\begin{enumerate}[label=\textnormal{(\alph*)}]
\item If $f_1$ and $f_2$ are bell-shaped functions corresponding to parameters $a_1, b_1, c_1, \ph_1$ and $a_2, b_2, c_2, \ph_2$ in Theorem~\ref{thm:bell}, and the convolution $f_1 * f_2$ is well-defined, then the Fourier transform of $f_1 * f_2$ has the representation~\eqref{eq:bell}, with $a = a_1 + a_2$, $b = b_1 + b_2$, $c = c_1 + c_2$ and $\ph = \ph_1 + \ph_2$. Note that while the integrability condition~\ref{thm:bell:b} is automatically satisfied, $\ph$ may fail to satisfy the level-crossing condition~\ref{thm:bell:a} and the regularity condition~\ref{thm:bell:c}.
\item Given a function $f$, it is relatively easy to verify whether $f$ is of the form given by Theorem~\ref{thm:bell}. Indeed, if $f$ is as in Theorem~\ref{thm:bell}, the parameters $a$, $b$, $c$ and $\ph$ can be recovered from the Fourier transform $\Phi(\xi) = \laplace f(i \xi)$ of $f$ in the following way. The function $\Phi$ extends to a zero-free holomorphic function in the right complex half-plane $\{\xi \in \C : \re \xi > 0\}$, and $\ph(s) ds$ is the vague limit of $\pi^{-1} \im \log \Phi(t + i s) ds$ as $t \to 0^+$, where $\log \Phi$ is a continuous version of the complex logarithm of $\Phi$; see Remark~5.5 in~\cite{bell}. Having determined $\ph$, it is straightforward to identify the parameters $a$, $b$ and $c$.
\end{enumerate}
\end{remarks}

As explained in Section~6 of~\cite{bell}, the class of functions described by the above result may seem artificial. More precisely, the level-crossing condition~\ref{thm:bell:a} appears rather unnatural. However, to a great surprise of the authors, it turns out that Theorem~\ref{thm:bell} describes \emph{all} bell-shaped functions. This is the main result of the present article.

\begin{theorem}
\label{thm:bell2}
Every weakly bell-shaped function is of the form described in Theorem~\ref{thm:bell}.
\end{theorem}

The above statement can be thought of as the bell-shape analogue of the classical Bernstein's theorem, which identifies completely monotone functions with Laplace transforms of non-negative measures.

The following two direct corollaries of Theorem~\ref{thm:bell2} are of probabilistic nature.

\begin{corollary}
\label{cor:roots}
Every weakly (resp., strictly) bell-shaped probability distribution function is infinitely divisible, and its convolution roots are weakly (resp., strictly) bell-shaped, too.
\end{corollary}

\begin{corollary}
\label{cor:walk}
Suppose that $X_n$ is a random walk in $\R$ (that is, $X_0 = 0$ and $X_{n+1} - X_n$ is an i.i.d. sequence of random variables). The following two conditions are equivalent:
\begin{enumerate}[label=\textnormal{(\alph*)}]
\item\label{cor:walk:a} the distribution of $X_n$ is weakly bell-shaped for every $n = 1, 2, \ldots$;
\item\label{cor:walk:b} the distribution of $X_1$ is an extended generalised gamma convolution, that is, it has a representation as in Theorem~\ref{thm:bell}, with $\ph$ a non-decreasing function.
\end{enumerate}
\end{corollary}

The next result follows from Theorem~\ref{thm:bell2} in a less straightforward way. It provides information about the distribution of the zeroes of the derivatives of bell-shaped functions. For simplicity, throughout this article by a zero of the derivative $f^{(n)}$ of a bell-shaped function $f$ we always understand a point at which $f^{(n)}$ changes its sign.

\begin{corollary}
\label{cor:zeroes}
If $f$ is a bell-shaped (or a weakly bell-shaped) function, then there are constants $p, q$ such that for every $n = 1, 2, \ldots$ all zeroes of~$f^{(n)}$ are contained in $[p n, q n]$.
\end{corollary}

Therefore, the zeroes of the derivative $f^{(n)}$ of a bell-shaped function $f$ move away from the origin at most linearly with $n$ as $n \to \infty$. A more refined result states that with the notation of Theorem~\ref{thm:bell}, the rates at which zeroes of the derivatives of $f$ diverge, in fact describe the parameter $a$ and the points $s_k$ at which $\ph - k$ changes its sign for $k \in \Z$ (see Figure~\ref{fig:phi}). A rigorous formulation of this result, however, requires additional definitions and a brief discussion of the proof of Theorem~\ref{thm:bell}. These are the main subject of the following two subsections.

\subsection{Pólya frequency functions, $\amcm$ functions and factorisation of bell-shaped functions}
\label{sec:intro:decomp}

Since a weakly bell-shaped function $f$ is a vague limit of strictly bell-shaped functions, every weakly bell-shaped function is absolutely continuous with respect to the Lebesgue measure, except possibly a single point, where it may contain an atom of non-negative mass. Therefore, weakly bell-shaped functions are essentially genuine functions, possibly with an additional atom. To keep the notation more intuitive, we follow the convention introduced in~\cite{bell}, and we call such measures \emph{extended functions}. Thus, an extended function $f$ is a measure which is absolutely continuous with respect to the Lebesgue measure, except possibly at a single point. We denote the density function of $f$ by the same symbol $f(x)$, and if $f$ has an atom at $p$, its mass is denoted by $f(\{p\})$.

Also following~\cite{bell}, we say that an extended function $g$ is \emph{absolutely monotone-then-completely monotone}, or $\amcm$ in short, if $g(x)$ and $g(-x)$ are completely monotone functions of $x > 0$, and the atom of $g$, if present, is located at $0$ and has non-negative mass. Bernstein's theorem asserts that $g$ is an $\amcm$ function if and only if there are non-negative measures $\mu_+$ and $\mu_-$ on $[0, \infty)$ such that $g(x) = \laplace \mu_+(x)$ for $x > 0$ and $g(x) = \laplace \mu_-(-x)$ for $x < 0$; an extended $\amcm$ function may contain an additional atom at $0$. We say that $\mu_+$ and $\mu_-$ are Bernstein measures of $g$.

We have the following characterisation of $\amcm$ functions.

\begin{figure}
\centering
\includegraphics[width=0.8\textwidth]{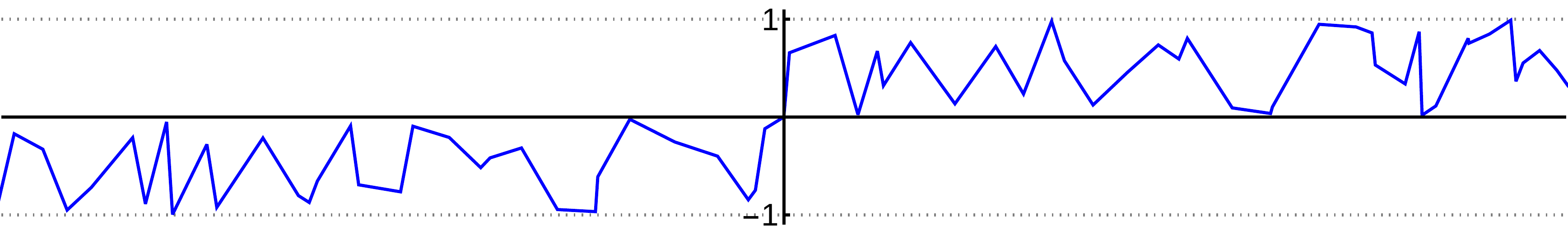}
\caption{A sample function $\ph$ in Proposition~\ref{prop:amcm}.}
\label{fig:amcm-phi}
\end{figure}

\begin{proposition}[Corollary~3.3 and Proposition~5.1 in~\cite{bell}]
\label{prop:amcm}
If $g$ is a locally integrable extended function, possibly with an atom at $0$ of non-negative mass, $g$ converges to zero at $\pm\infty$ and $g$ is non-decreasing near $-\infty$ and non-increasing near $\infty$, then the following conditions are equivalent:
\begin{enumerate}[label=\textnormal{(\alph*)}]
\item\label{it:amcm} $g$ is an $\amcm$ function;
\item\label{it:amcm:int} the Fourier transform of $g$ is of the form
\formula[eq:amcm:int]{
 \laplace g(i \xi) & = m + \int_{(0, \infty)} \frac{1}{i \xi + s} \, \mu_+(ds) - \int_{(0, \infty)} \frac{1}{i \xi - s} \, \mu_-(ds)
}
for $\xi \in \R \setminus \{0\}$, where $m \ge 0$ and $\mu_+$ and $\mu_-$ are non-negative measures on $(0, \infty)$ such that the above integrals are finite; conversely, any such $m$, $\mu_+$ and $\mu_-$ correspond to some $\amcm$ function $g$; here $m = g(\{0\})$ and $\mu_+$, $\mu_-$ are the Bernstein measures of~$g$;
\item\label{it:amcm:exp} the Fourier transform of $g$ is of the form
\formula[eq:amcm:exp]{
 \laplace g(i \xi) & = \exp\biggl(c + \int_{-\infty}^\infty \biggl( \frac{1}{i \xi + s} - \frac{1}{s} \, \ind_{\R \setminus (-1, 1)}(s) \biggr) \ph(s) ds \biggr)
}
for $\xi \in \R \setminus \{0\}$, where $c \in \R$, $\ph : \R \to \R$ is a Borel function such that $\ph(s) \sign s \in [0, 1]$ for almost all $s \in \R$, and condition~\ref{thm:bell:c} of Theorem~\ref{thm:bell} is satisfied (see Figure~\ref{fig:amcm-phi}); conversely, any such $c$ and $\ph$ correspond to some $\amcm$ function.
\end{enumerate}
\end{proposition}

Note that condition~\ref{it:amcm:exp} expresses the Fourier transform of $g$ as in~\eqref{eq:bell}, with $a = 0$, appropriately chosen $b$, and $\ph$ such that $\ph(s) \sign s \in [0, 1]$ for almost all $s \in \R$.

A function $h$ is said to be a \emph{Pólya frequency function} if and only if it is the density function of a probability distribution, which is a weak limit of convolutions of exponential distributions (and their mirror images). This is the classical definition, and a number of equivalent variants exist, some of which are collected in the following statement. We remark that we will not use item~\ref{it:pff:tp}, which is only given for reader's convenience. For further details and detailed discussion, we refer to the monograph~\cite{karlin}, as well as to the original works of Schoenberg~\cite{schoenberg1,schoenberg2}; item~\ref{it:pff:exp} is taken from Proposition~5.3 in~\cite{bell}.

\begin{figure}
\centering
\includegraphics[width=0.8\textwidth]{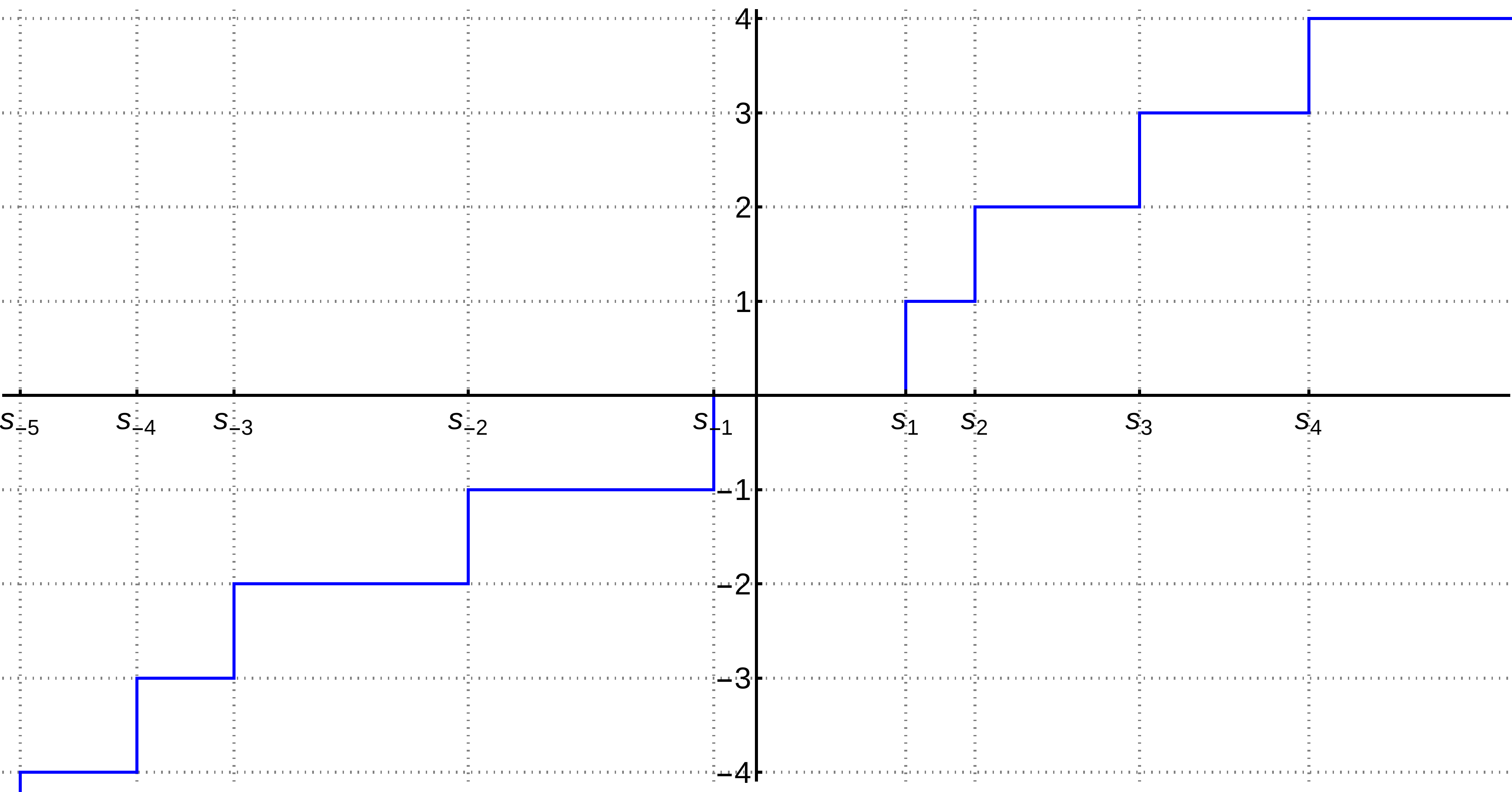}
\caption{A sample function $\ph$ in Proposition~\ref{prop:pff}.}
\label{fig:pff-phi}
\end{figure}

\begin{proposition}
\label{prop:pff}
The following conditions are equivalent:
\begin{enumerate}[label=\textnormal{(\alph*)}]
\item\label{it:pff} $h$ is a Pólya frequency function;
\item\label{it:pff:int} $h$ is integrable, and the Fourier transform of $h$ is given by
\formula[eq:pff]{
 \laplace h(i \xi) & = e^{-a \xi^2 - i b \xi} \prod_{k = 1}^N \frac{e^{i \alpha_k \xi}}{1 + i \alpha_k \xi}
}
for $\xi \in \R$, where $a \ge 0$, $b \in \R$, $N \in \{0, 1, 2, \ldots, \infty\}$ and $\alpha_k \in \R \setminus \{0\}$ for every $k$, and we assume that $\sum_{k = 1}^N |\alpha_k|^2 < \infty$; conversely, any such parameters $a$, $b$, $N$ and $\alpha_k$ correspond to a Pólya frequency function;
\item\label{it:pff:exp} $h$ is integrable, and the Fourier transform of $h$ is of the form~\eqref{eq:bell} for $\xi \in \R$, where $a \ge 0$, $b \in \R$, $c = 0$, $\ph$ is non-decreasing and only takes integer values, $\ph(s) = 0$ in a neighbourhood of $0$, and integrability condition~\ref{thm:bell:b} of Theorem~\ref{thm:bell} is satisfied (see Figure~\ref{fig:pff-phi}); conversely, any such parameters $a$, $b$ and $\ph$ correspond to a Pólya frequency function;
\item\label{it:pff:kernel} $h$ is integrable, with integral $1$, and it is a \emph{variation-diminishing convolution kernel}: for every bounded Borel function $f$, the convolution $f * h$ changes its sign no more times than $f$ does;
\item\label{it:pff:tp} $h$ is integrable, with integral $1$, and $h(x - y)$ is a totally positive kernel.
\end{enumerate}
\end{proposition}

A detailed discussion of Pólya frequency functions can be found in~\cite{hw}. We note that there is a close connection between the parameters $\alpha_k$ and the function $\ph$: the numbers $s_k = 1 / \alpha_k$ are the locations of the jumps of $\ph$, repeated according to the height of the jump. More precisely, $\ph$ is the distribution function, normalised by the condition $\ph(0) = 0$, of the measure $\sum_{k = 1}^N \delta_{1 / \alpha_k}(ds)$.

The main idea behind the proof of Theorem~\ref{thm:bell} in~\cite{bell} is as follows: if $f$ satisfies the assumptions of the theorem, then $f$ can be expressed as the convolution of $g$ and $h$, where $g$ is a locally integrable $\amcm$ extended function which converges to zero at $\pm \infty$, and $h$ is a Pólya frequency function. It is then proved that locally integrable $\amcm$ extended functions which converge to zero at $\pm \infty$ are weakly bell-shaped. By condition~\ref{it:pff:kernel} in Proposition~\ref{prop:pff}, $f$ is weakly bell-shaped. Furthermore, a smooth weakly bell-shaped function is automatically strictly bell-shaped. We may therefore re-phrase Theorems~\ref{thm:bell} and~\ref{thm:bell2} as follows.

\begin{corollary}
\label{cor:bell}
The following conditions are equivalent:
\begin{enumerate}[label=\textnormal{(\alph*)}]
\item $f$ is a weakly bell-shaped extended function;
\item $f$ is the convolution of a Pólya frequency function and a locally integrable $\amcm$ extended function which converges to zero at $\pm \infty$;
\item $f$ is a locally integrable extended function which converges to zero at $\pm \infty$, which is non-decreasing near $-\infty$ and non-increasing near $\infty$, the Fourier transform of $f$ is given by~\eqref{eq:bell}, and conditions~\ref{thm:bell:a}, \ref{thm:bell:b} and~\ref{thm:bell:c} of Theorem~\ref{thm:bell} are satisfied.
\end{enumerate}
\end{corollary}

As described in detail in Lemma~5.4 in~\cite{bell}, under the assumptions of Theorem~\ref{thm:bell} the decomposition $f = g * h$ of $f$ into the convolution of an $\amcm$ extended function $g$ and a Pólya frequency function $h$ is relatively easy. Indeed, one constructs an integer-valued, non-decreasing function $\ph_h(s)$ so that $\ph_h(s) \le \ph(s) \le \ph_h(s) + 1$ for $s > 0$ and $\ph_h(s) - 1 \le \ph(s) \le \ph_h(s)$ for $s < 0$ (this is possible by virtue of the level-crossing condition~\ref{thm:bell:a} in Theorem~\ref{thm:bell}), and then one defines $\ph_g(s) = \ph(s) - \ph_h(s)$. By construction, $\ph_g(s) \sign s \in [0, 1]$ for every $s \in \R$. If $g$ is an $\amcm$ extended function associated with parameters $c$ and $\ph_g$, and $h$ is a Pólya frequency function corresponding to parameters $a$, appropriately modified $b$ and $\ph_h$, then $f = g * h$, as desired.

Recall that $\ph_h$ has jump discontinuities at points at which $\ph - k$ changes its sign. We stress that if the function $\ph - k$ is equal to zero on an interval for some $k \in \Z \setminus \{0\}$, then more than one choice of $\ph_h$ is possible, and thus the decomposition $f = g * h$ is not unique. However, if $\ph - k$ changes its sign at at most one point for every $k \in \Z \setminus \{0\}$, then the factorisation $f = g * h$ is indeed unique.

While in~\cite{bell} exponential representations of Fourier transforms of $g$ and $h$ (given in item~\ref{it:amcm:exp} of Proposition~\ref{prop:amcm} and item~\ref{it:pff:exp} of Proposition~\ref{prop:pff}) played a crucial role, here we equally often refer to condition~\ref{it:amcm:int} in Proposition~\ref{prop:amcm} and condition~\ref{it:pff:int} in Proposition~\ref{prop:pff}.

\subsection{Zeroes of derivatives of bell-shaped functions}
\label{sec:intro:bs2}

\begin{figure}
\centering\scriptsize
\begin{tabular}{ccc}
\includegraphics[width=0.3\textwidth]{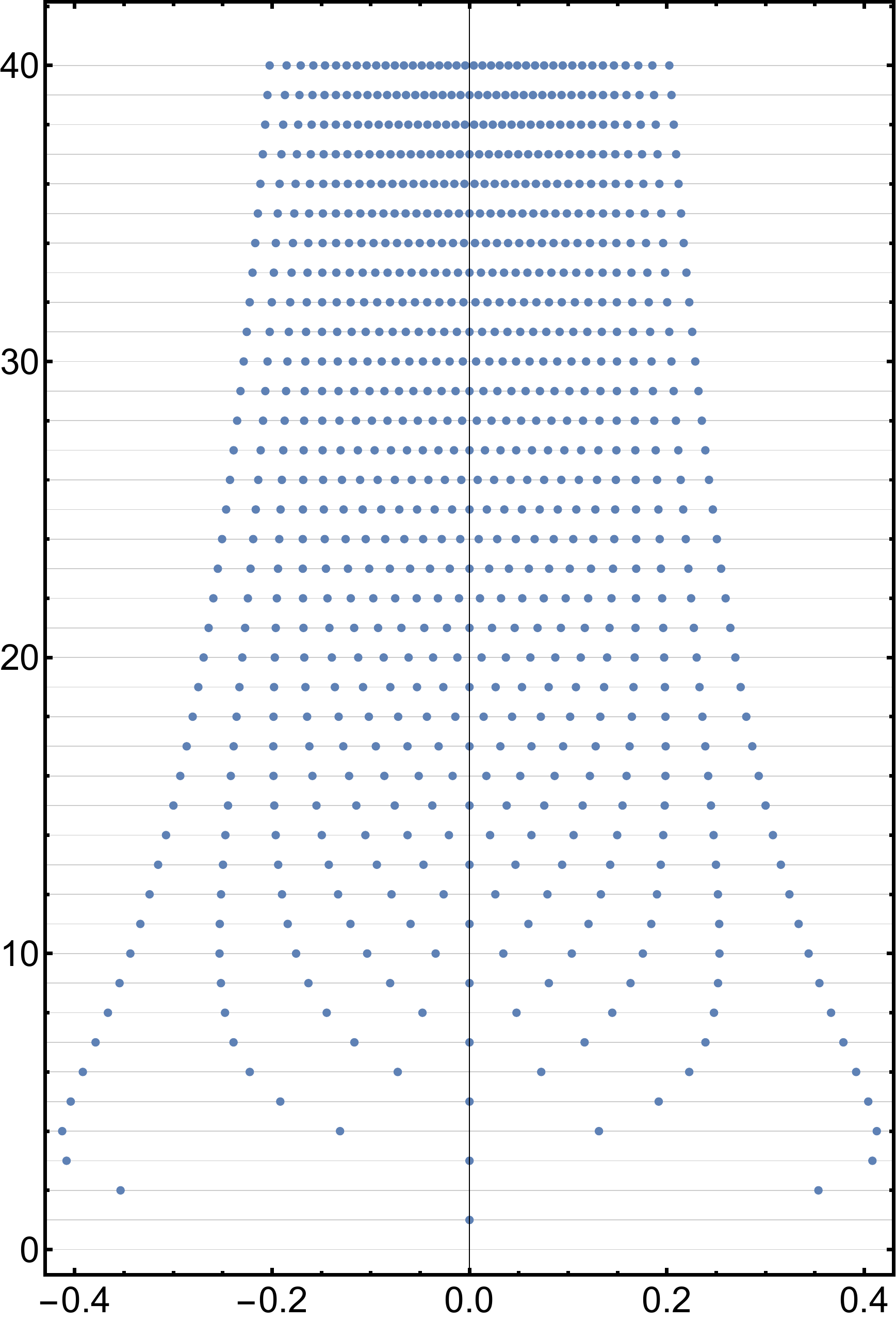} &
\includegraphics[width=0.3\textwidth]{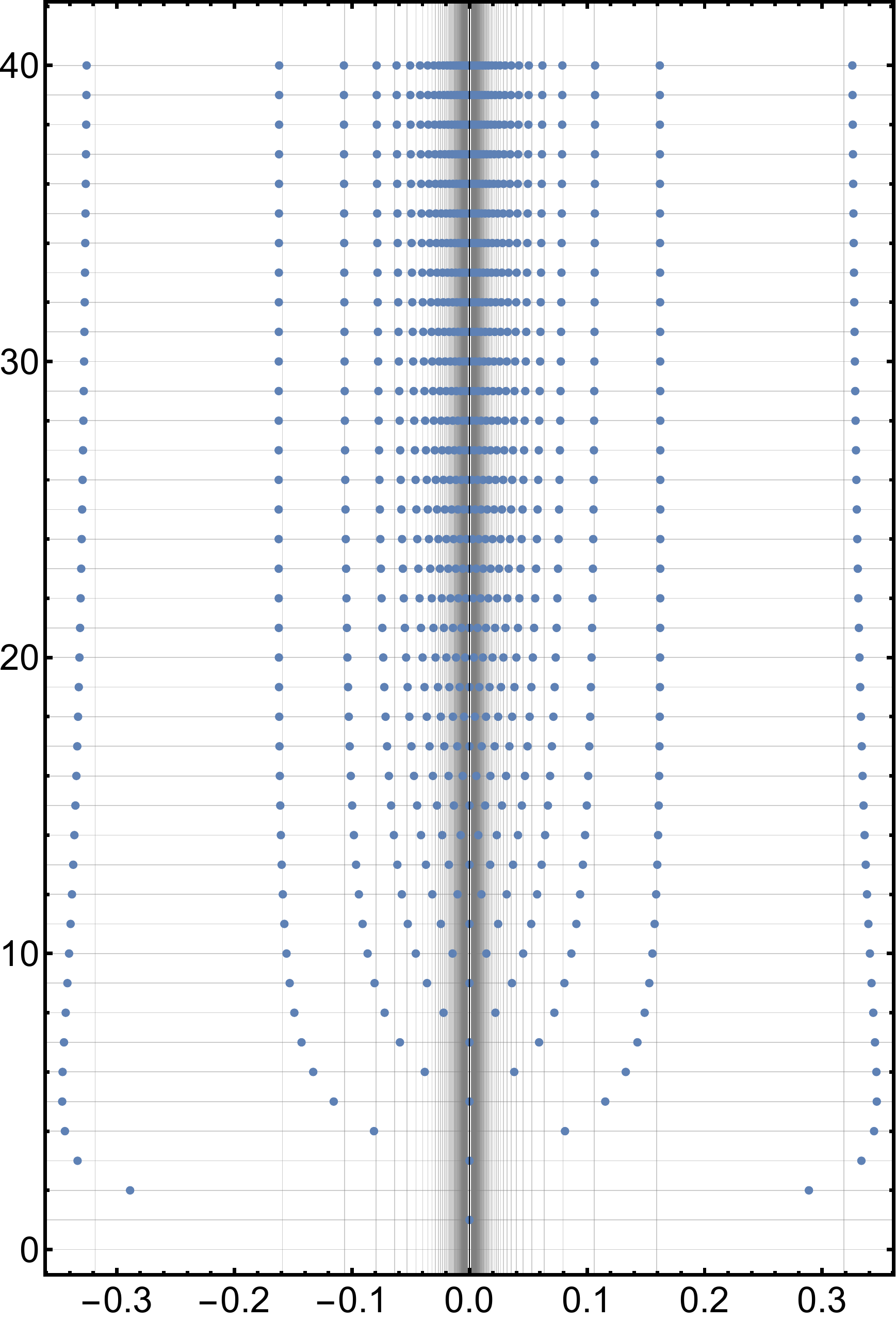} &
\includegraphics[width=0.3\textwidth]{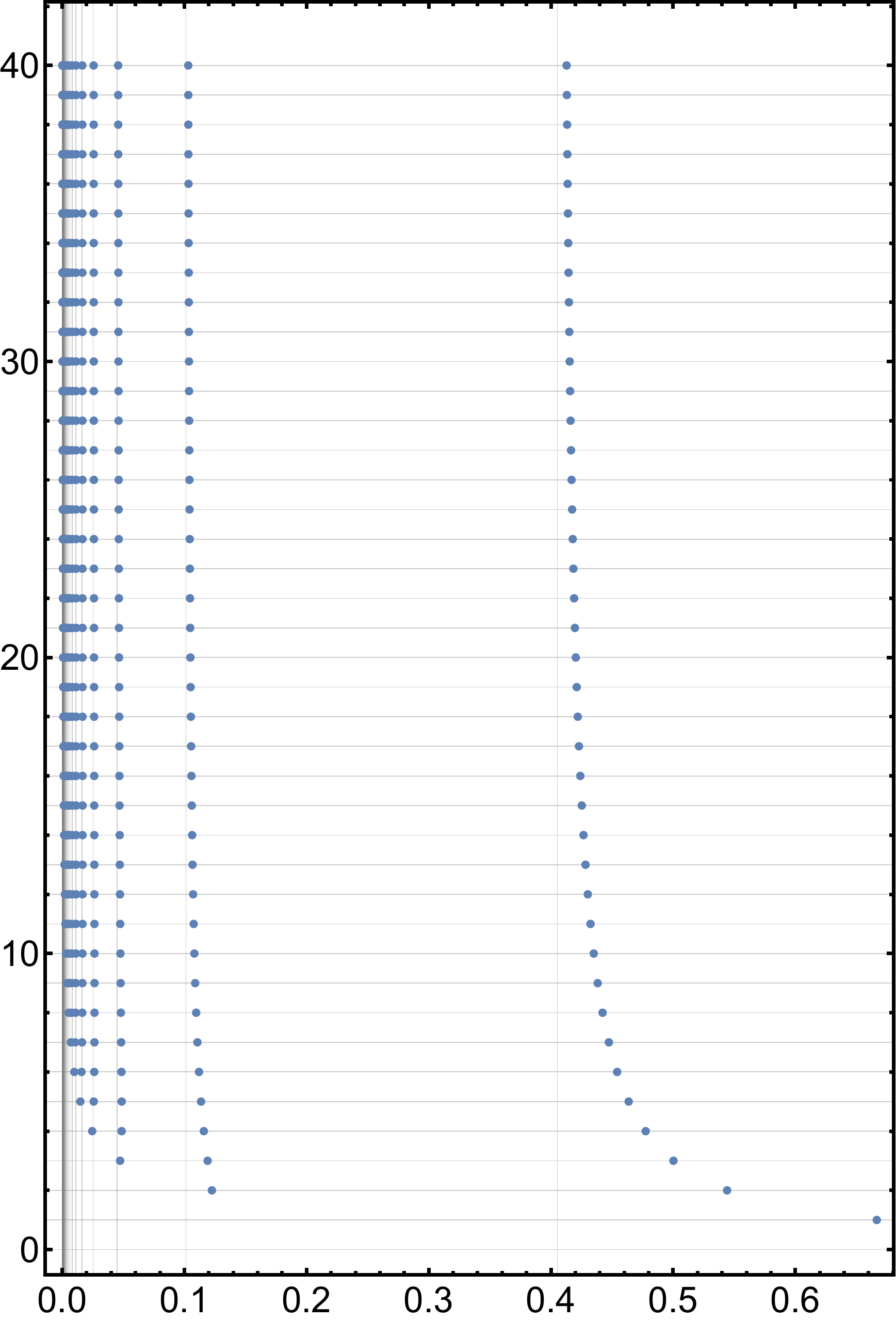} \\
(a) & (b) & (c)
\end{tabular}
\caption{Location of the zeroes of $f^{(n)}(n x)$ (horizontal axis) versus $n = 0, 1, 2, \ldots, 40$ (vertical axis) for: (a)~$f(x) = \exp(-x^2)$ (the density function of the normal distribution); (b)~$f(x) = (1 + x^2)^{-1}$ (the density function of the Cauchy distribution); (c)~$f(x) = x^{-3/2} \exp(-1/x)$ (the density function of the Lévy distribution). Grid lines indicate limiting positions $\alpha_k = 1 / s_k$ of the zeroes: (a)~all zeroes accumulate near $0$; (b)~$\alpha_k = 1 / (k \pi)$, $k \in \Z \setminus \{0\}$; (c)~$\alpha_k = 4 / (k^2 \pi^2)$, $k \in \{1, 2, \ldots\}$.}
\label{fig:zeroes}
\end{figure}

\begin{figure}
\centering\scriptsize
\begin{tabular}{ccc}
\includegraphics[width=0.45\textwidth]{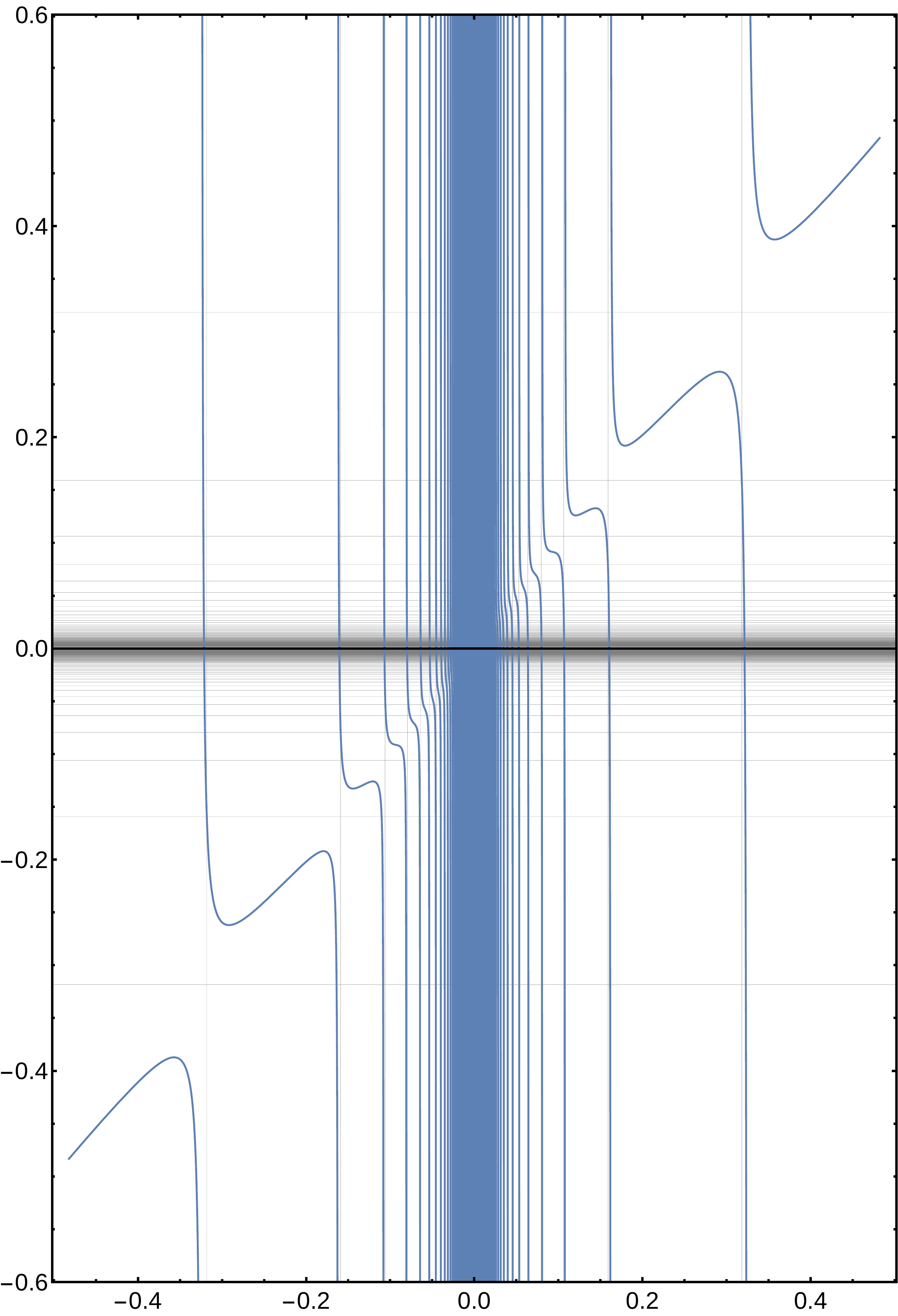} &
\includegraphics[width=0.45\textwidth]{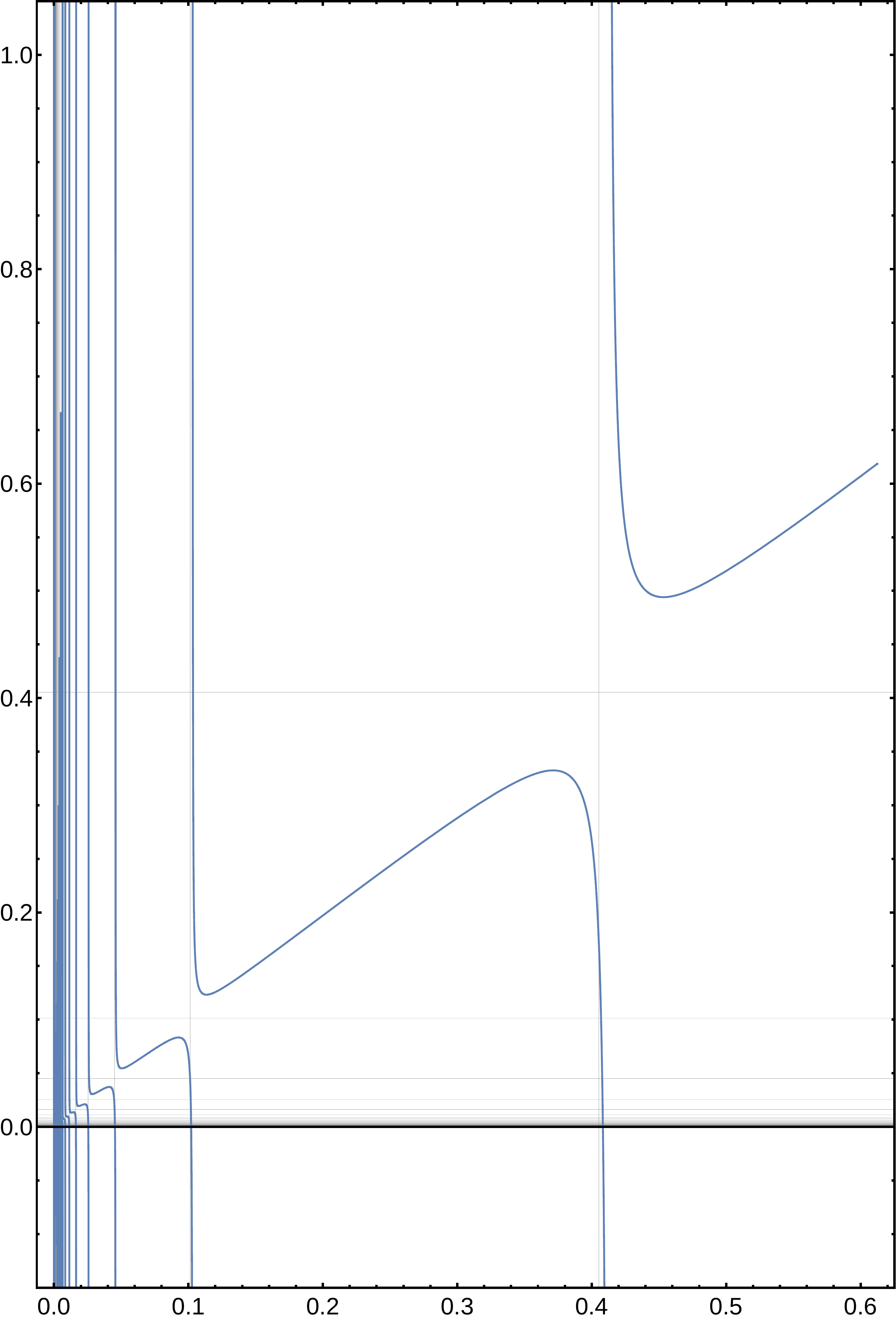} \\
(a) & (b)
\end{tabular}
\caption{Location of the zeroes of $f^{(n)}_p(n x)$ for $f_p(x) = f(x) + p f'(x)$ (horizontal axis) versus parameter $p$ (vertical axis), with $n = 100$ and: (a)~$f(x) = (1 + x^2)^{-1}$ (the density function of the Cauchy distribution); (b)~$f(x) = x^{-3/2} \exp(-1/x)$ (the density function of the Lévy distribution). Horizontal and vertical grid lines indicate parameters $p$ for which $f_p$ is bell-shaped, and at the same time the limiting positions $\alpha_k = 1 / s_k$ of the zeroes of $f^{(n)}(n x)$ (as in Figure~\ref{fig:zeroes})}: (a)~$\alpha_k = 1 / (k \pi)$, $k \in \Z \setminus \{0\}$; (b)~$p = 4 / (k^2 \pi^2)$, $k \in \{1, 2, \ldots\}$.
\label{fig:derivatives}
\end{figure}

Suppose that $f = g * h$ is the decomposition of a strictly bell-shaped function into the convolution of an $\amcm$ function $g$ and a Pólya frequency function $h$. For simplicity, below we call such a decomposition a \emph{canonical factorisation} of $f$. Interestingly, the parameters $a$ and $\alpha_k$ in the representation~\eqref{eq:pff} of the Pólya frequency function $h$ are (essentially) determined by the zeroes of the derivatives of $f$; see Figure~\ref{fig:zeroes}. Recall that $s_k = 1 / \alpha_k$ are the points at which the corresponding function $\ph_h$ (described in item~\ref{it:pff:exp} in Proposition~\ref{prop:pff}; see Figure~\ref{fig:phi}) has jump discontinuities.

\begin{proposition}
\label{prop:zeroes}
Let $f$ be a strictly bell-shaped function, and denote the zeroes of $f^{(n)}$ by $n \alpha_{n,k}$, $k = 1, 2, \ldots, n$. Then the sequence of measures
\formula[eq:zmn]{
 \sum_{k = 1}^n \alpha_{n,k}^2 \delta_{\alpha_{n,k}}(dx) 
}
is relatively compact with respect to the topology of weak convergence. In particular, all numbers $\alpha_{n,k}$ are uniformly bounded. The partial limits of the sequence~\eqref{eq:zmn} are of the form
\formula[eq:zm]{
 2 a \delta_0(dx) + \sum_{k = 1}^N \alpha_k^2 \delta_{\alpha_k}(dx) ,
}
where $\alpha_k \ne 0$ is a (finite or infinite) sequence of points, and $a \ge 0$. In this case there is a canonical factorisation $f = g * h$ of $f$ such that the Fourier transform of $h$ has the representation~\eqref{eq:pff} with the parameters $a$ and $\alpha_k$ defined above. Consequently, the function $\ph$ in the representation~\eqref{eq:bell} of the Fourier transform of $f$ crosses integer levels at the family of points $1 / \alpha_k$.
\end{proposition}

\begin{remark}
We emphasize that if the function $\ph$ in the representation~\eqref{eq:bell} of the Fourier transform of $f$ takes an integer value on an interval, then the canonical factorisation $f = g * h$ is not unique. In this case it is not clear whether the sequence of measures~\eqref{eq:zmn} actually converges: we conjecture this to be true, but we were unable to prove this. If, however, $\ph - k$ changes its sign at at most one point $s_k$ for every $k \in \Z \setminus \{0\}$, then the canonical factorisation is unique, and consequently the sequence of measures~\eqref{eq:zmn} in Proposition~\ref{prop:zeroes} converges weakly to the limit~\eqref{eq:zm}, with the sequence $\alpha_k$ being a rearrangement of the sequence $1 / s_k$. In this case, if $0 < x_1 < x_2$ or $x_1 < x_2 < 0$ and neither $x_1$ nor $x_2$ is equal to any of the numbers $\alpha_k$, then
\formula{
 \lim_{n \to \infty} \# \{k : x_1 < \alpha_{n,k} < x_2\} & = \# \{k : x_1 < \alpha_k < x_2\} ,
}
that is, the numbers $\alpha_j$ coincide with non-zero limits of sequences of the form $\alpha_{n,k(n)}$.
\end{remark}

Proposition~\ref{prop:zeroes} shows that the location of zeroes of the derivatives of a strictly bell-shaped function $f$ essentially describes the factor $h$ in the canonical factorisation $f = g * h$ of $f$. It is an open question whether they also describe the factor $g$. We conjecture that this indeed the case; in other words, if $f_1$ and $f_2$ are two strictly bell-shaped functions such that $f_1^{(n)}(x) = 0$ if and only if $f_2^{(n)}(x) = 0$ for all $n = 1, 2, \ldots$ and $x \in \R$, then $f_1 / f_2$ is constant.

We conclude this section with the following observation. Suppose that $f$ is a strictly bell-shaped function described by Theorem~\ref{thm:bell}. If $k \in \Z \setminus \{0\}$ and $\ph - k$ changes its sign at $1 / p$, then $\ph_p$ defined by
\formula{
 \ph_p(s) & = \begin{cases} \ph(s) - \ind_{[1 / p, \infty)}(s) & \text{if $p > 0$} \\ \ph(s) + \ind_{(-\infty, 1 / p]}(s) & \text{if $p < 0$} \end{cases}
}
satisfies all conditions of Theorem~\ref{thm:bell}. Conversely, if $\ph_p$ satisfies the level-crossing condition~\ref{thm:bell:a} of Theorem~\ref{thm:bell}, then necessarily $\ph - k$ changes its sign at $1 / p$ for some $k \in \Z \setminus \{0\}$. Furthermore, it is easy to see that the parameters $a$, $\ph_p$, $c_p = c + \log |p|$ and appropriately modified $b_p$ correspond in Theorem~\ref{thm:bell} to the function
\formula{
 f_p(x) & = f(x) + p f'(x) .
}
It follows that $f_p$ is (strictly) bell-shaped if and only if $\ph - k$ changes its sign at $1 / p$ for some $k \in \Z \setminus \{0\}$. This proves the following surprising result.

\begin{corollary}
\label{cor:der}
If $f$ is a strictly bell-shaped function, $p \in \R \setminus \{0\}$ and $f_p(x) = f(x) + p f'(x)$, then $f$ is strictly bell-shaped if and only if the function $\ph$ corresponding to $f$ in Theorem~\ref{thm:bell} crosses some non-zero integer level at $1 / p$.

In particular, if $\ph - k$ changes its sign at at most one point $s_k$ for every $k \in \Z \setminus \{0\}$, then $f_p$ is strictly bell-shaped if and only if $p = 0$ or $p = 1 / s_k$ for some $k$; or, equivalently, there is a sequence $\alpha_n$ such that $n \alpha_n$ is a zero of $f^{(n)}$ and $p = \lim_{n \to \infty} (1 / \alpha_n)$.
\end{corollary}

We illustrate the above corollary by three examples. If $f$ is the density function of a normal distribution, then $f_p$ is never bell-shaped. This is quite clear: in fact, $f_p$ is never positive. If $f(x) = 1 / (1 + x^2)$, then
\formula{
 f_p(x) & = \frac{1 - 2 p x + x^2}{(1 + x^2)^2}
}
is bell-shaped if and only if $p = 1 / (k \pi)$ for some $k \in \Z \setminus \{0\}$. Similarly, if $f(x) = e^{-1/x} \ind_{(0, \infty)}(x)$, then
\formula{
 f_p(x) & = \biggl( 1 + \frac{p}{x^2} \biggr) e^{-1/x}
}
is bell-shaped if and only if $p = 4 / (k^2 \pi^2)$ for some $k \in \{1, 2, \ldots\}$. The last two examples do not appear to have a simple, elementary derivation.

\subsection{Whale-shaped functions}
\label{sec:intro:ws}

A function $f$ on $(0, \infty)$ is said to be \emph{whale-shaped} if $f$ is positive and smooth, $f$ converges to zero at $0$ and $\infty$, and $f^{(n)}$ changes its sign only once for $n = 1, 2, \ldots$\, More generally, we say that a function $f$ on $(0, \infty)$ is \emph{whale-shaped of order $m \in \{0, 1, 2, \ldots\}$} if $f$ is positive, smooth, $f$ converges to zero at $\infty$, $f^{(n)}$ converges to zero at $0$ for $n = 0, 1, 2, \ldots, m - 1$, and $f^{(n)}$ changes its sign $\min\{n, m\}$ times for $n = 0, 1, 2, \ldots$.\, In particular, $f$ is whale-shaped if it is whale-shaped of order $1$, and $f$ is completely monotone if $f$ is whale-shaped of order $0$. The notion of whale-shaped functions was introduced in~\cite{simon} under the name \emph{weakly bell-shaped functions of order $m$}, where the direct half of the following result was proved. The present name originates in~\cite{js}.

\begin{theorem}
\label{thm:whale}
For $m = 0, 1, 2, \ldots$, the following conditions are equivalent:
\begin{enumerate}[label=\textnormal{(\alph*)}]
\item\label{thm:whale:a} $f$ is a whale-shaped function of order $m$, integrable near $0$ if $m = 0$;
\item\label{thm:whale:b} $f$ is the convolution of $m$ exponential factors of the form $\alpha_j^{-1} \exp(-x / \alpha_j)$, $\alpha_j > 0$, $j = 1, 2, \ldots, m$, and a completely monotone function on $(0, \infty)$ which converges to zero at $\infty$ and which is integrable near $0$.
\end{enumerate}
In particular, every whale-shaped function of order $m$ (integrable near $0$ if $m = 0$) is weakly bell-shaped.
\end{theorem}

\subsection{Discussion}
\label{sec:intro:dis}

The proof of Theorem~\ref{thm:bell2}, at least when integrable bell-shaped functions $f$ are considered, is surprisingly elementary. It combines an application of Post's formula for the inverse Laplace transform of the Cauchy--Stieltjes transform of $f$ with some ideas developed by Hirschman in~\cite{hirschman}. Our proof of Theorem~\ref{thm:bell2} has its roots in Proposition~10 in~\cite{hsw}, which asserts that whale-shaped functions have the representation similar to that in Theorem~\ref{thm:bell}, with $\ph$ taking values in $[0, 2]$.

The structure of the remaining part of this article is as follows. For clarity, in Section~\ref{sec:bs1} we give a simplified proof of Theorem~\ref{thm:bell2} for integrable bell-shaped functions. Section~\ref{sec:bs2} contains the full proof in the general case. Additional results on bell-shaped functions are discussed in Section~\ref{sec:bs3}, while Section~\ref{sec:ws} briefly sketches the proof of Theorem~\ref{thm:whale} on whale-shaped functions.

Finally, we discuss briefly the notation used throughout the article. As usual, we say that a sequence of Borel measures $\mu_n$ on $\R$ (or a similar space) converges weakly to a measure $\mu$ if $\int_\R f(x) \mu_n(dx)$ converges to $\int_\R f(x) \mu(dx)$ for every bounded and continuous $f$. Similarly, the sequence $\mu$ converges vaguely to $\mu$ if $\int_\R f(x) \mu_n(dx)$ converges to $\int_\R f(x) \mu(dx)$ for all continuous $f$ with compact support. We denote by $\laplace f(\xi) = \int_{-\infty}^\infty e^{-\xi x} f(x) dx$ the two-sided Laplace transform of $f$ (here $\xi \in \C$), and we re-use this notation for the Fourier transform of $f$, $\laplace f(i \xi)$ (here $\xi \in \R$). The Laplace and Fourier transforms of a measure $\mu$ are defined in a similar way. These definitions are extended to the case when $f$ or $\mu$ is non-integrable at $\pm \infty$, provided that the integrals in the definitions of $\laplace f$ or $\laplace \mu$ are well-defined as improper integrals.

%
%                            ---------- o ----------
%

\section{Bell-shaped functions, integrable case}
\label{sec:bs1}

In this section we sketch the proof of Theorem~\ref{thm:bell2} for \emph{integrable} bell-shaped functions. In this case the argument avoids certain technicalities, and the idea of the proof is easier to follow. To further facilitate reading of this section, we postpone the discussion of some additional details to the next section, where the general case is studied.

For reader's convenience, before we prove Theorem~\ref{thm:bell2}, we first re-phrase Theorem~\ref{thm:bell} for integrable functions. Suppose that $f$ is a weakly bell-shaped function described by Theorem~\ref{thm:bell}; by saying this, we mean that the Fourier transform of $f$ is given by~\eqref{eq:bell} for $\xi \in \R \setminus \{0\}$, and conditions~\ref{thm:bell:a}, \ref{thm:bell:b} and~\ref{thm:bell:c} of the theorem are satisfied. It follows that the continuous version of the complex logarithm of $\laplace f(i \xi)$ is given by
\formula[eq:bell:log]{
 \log \laplace f(i \xi) & = -a \xi^2 - i b \xi + c + \int_{-\infty}^\infty \biggl( \frac{1}{i \xi + s} - \biggl(\frac{1}{s} - \frac{i \xi}{s^2} \biggr) \ind_{\R \setminus (-1, 1)}(s) \biggr) \ph(s) ds .
}
It is convenient to write down the expressions for the real and imaginary parts of the above expression:
\formula[]{
\label{eq:bell:rlog}
 \re \log \laplace f(i \xi) & = -a \xi^2 + c + \int_{-\infty}^\infty \biggl( \frac{s}{\xi^2 + s^2} - \frac{1}{s} \, \ind_{\R \setminus (-1, 1)}(s) \biggr) \ph(s) ds , \\
\label{eq:bell:ilog}
 \im \log \laplace f(i \xi) & = -b \xi - \xi \int_{-\infty}^\infty \biggl( \frac{1}{\xi^2 + s^2} - \frac{1}{s^2} \, \ind_{\R \setminus (-1, 1)}(s) \biggr) \ph(s) ds .
}
Note that $\re \log \laplace f(i \xi)$ is simply equal to $\log |\laplace f(i \xi)|$, while $\im \log \laplace f(i \xi)$ is the continuous version of the complex argument of $\laplace f(i \xi)$.

If $f$ is integrable, then $\re \log \laplace f(i \xi)$ has a finite limit as $\xi \to 0$. The integral over $\R \setminus (-1, 1)$ in~\eqref{eq:bell:rlog} converges as $\xi \to 0$ to a finite limit by the dominated convergence theorem. On the other hand, the integral over $(-1, 1)$ has a (finite or infinite) limit $\int_{-1}^1 (\ph(s) / s) ds$ by the monotone convergence theorem. However, the left-hand side of~\eqref{eq:bell:rlog} has a finite limit, and so $\ph(s) / s$ is necessarily integrable near $s = 0$. Conversely, if $\ph(s) / s$ in integrable near zero, one can pass to a finite limit as $\xi \to 0$ in~\eqref{eq:bell}, which implies that $\laplace f(0)$ is finite, that is, $f$ is integrable. After simple rearrangement, we immediately find the following version of Theorems~\ref{thm:bell}.

\begin{corollary}
\label{cor:bell1}
Suppose that $a \ge 0$, $b, c \in \R$ and $\ph : \R \to \R$ is a Borel function with the following properties:
\begin{enumerate}[label=\textnormal{(\alph*)}]
\item\label{it:bell1:a} for every $k \in \Z$ the function $\ph(s) - k$ changes its sign at most once, and for $k = 0$ this change takes place at $s = 0$: we have $\ph(s) \ge 0$ for $s > 0$ and $\ph(s) \le 0$ for $s < 0$;
\item\label{it:bell1:b} we have
\formula{
 \int_{-\infty}^\infty \min(s^{-2}, s^{-4}) s \ph(s) ds < \infty .
}
\end{enumerate}
Then there is a weakly bell-shaped function $f$ such that for $\xi \in \R \setminus \{0\}$ the Fourier transform of $f$ satisfies
\formula[eq:bell1]{
 \laplace f(i \xi) & = \exp\biggl(-a \xi^2 - i b \xi - c - \int_{-\infty}^\infty \biggl(\frac{i \xi}{i \xi + s} - \frac{i \xi}{s} \, \ind_{\R \setminus (-1, 1)}(s) \biggr) \frac{\ph(s)}{s} \, ds \biggr) .
}
\end{corollary}

Note that while parameters $a$, $b$ and $\ph$ are the same in Theorem~\ref{thm:bell} and Corollary~\ref{cor:bell1}, the parameter $c$ may change. More precisely, the parameter $c$ in Corollary~\ref{cor:bell1} is equal to $-c - \int_{-1}^1 (\ph(s)/s) ds$ with the notation of Theorem~\ref{thm:bell}.

We remark that $f$ has integral at most one if and only if $\laplace f(0) = e^{-c} \le 1$, that is, $c \ge 0$ in Corollary~\ref{cor:bell1}. In this case, with the terminology of~\cite{rogers}, equation~\eqref{eq:bell1} means that $\laplace f(i \xi) = \exp(-\Phi(\xi))$, where $\Phi(\xi)$ is a \emph{Rogers function}: a holomorphic function in the right complex half-plane such that $\re(\Phi(\xi) / \xi) \ge 0$ for every $\xi$ in the right complex half-plane; see Theorem~3.3 in~\cite{rogers}, and Section~3 in~\cite{rogers} for a detailed discussion.

In this section we prove that every integrable weakly bell-shaped extended function is described by Corollary~\ref{cor:bell1}. We begin with three auxiliary lemmas. The first one contains a completely elementary property of bell-shaped functions; the second one is a statement about convergence of Fourier transforms of integrable bell-shaped functions, similar to Remark~3.16 in~\cite{bell}; the last one contains the key idea of the proof, based on Post's real inversion formula for the Laplace transform.

\begin{lemma}
\label{lem:fn:lim}
If $f$ is strictly bell-shaped, then
\formula[eq:fn:lim]{
 \lim_{x \to \pm \infty} (x^n f^{(n)}(x)) & = 0
}
for $n = 0, 1, 2, \ldots$\, If in addition $f$ is integrable, then in fact
\formula[eq:fn:lim:int]{
 \lim_{x \to \pm \infty} (x^{n + 1} f^{(n)}(x)) & = 0 .
}
\end{lemma}

\begin{proof}
Suppose that $f$ is strictly bell-shaped. We prove~\eqref{eq:fn:lim} by induction. The result for $n = 0$ is a part of the definition of a strictly bell-shaped function. Suppose now that~\eqref{eq:fn:lim} holds for some $n$; we will show that $x^{n + 1} f^{(n + 1)}(x)$ tends to zero as $x \to \pm \infty$.

Since $f$ is strictly bell-shaped, $f^{(n+2)}$ is non-negative in some neighbourhood of $-\infty$, and hence $f^{(n+1)}$ is non-decreasing in some neighbourhood of $-\infty$. It follows that if $x > 0$ is large enough, then 
\formula{
 2 f^{(n)}(-\tfrac{x}{2}) & = 2 \int_{-\infty}^{-x/2} f^{(n+1)}(y) dy \ge 2 \int_{-x}^{-x/2} f^{(n+1)}(y) dy \ge x f^{(n+1)}(-x) .
}
We conclude that $0 \le x^{n + 1} f^{(n+1)}(-x) \le 2 x^n f^{(n)}(-\tfrac{x}{2})$ for $x > 0$ sufficiently large. In a similar way, $0 \le (-x)^{n + 1} f^{(n+1)}(x) \le 2 (-x)^n f^{(n)}(\tfrac{x}{2})$ for $x > 0$ large enough, and consequently $x^{n + 1} f^{(n + 1)}(x)$ converges to zero as $x \to \pm \infty$, as desired.

If $f$ is an integrable strictly bell-shaped function, then in fact $x f(x)$ converges to zero as $x \to \pm \infty$. Indeed: $f$ is non-decreasing in some neighbourhood of $-\infty$, and hence for $x > 0$ large enough we have
\formula{
 0 \le x f(-x) & \le 2 \int_{-x}^{-x/2} f(y) dy .
}
As $x \to \infty$, the right-hand side converges to zero by the dominated convergence theorem, and hence $x f(-x)$ converges to zero as $x \to \infty$. Similarly, the limit of $x f(x)$ as $x \to \infty$ is zero, and~\eqref{eq:fn:lim:int} for $n = 0$ follows. The general case $n = 0, 1, 2, \ldots$ follows now by the induction argument used in the proof of~\eqref{eq:fn:lim}.
\end{proof}

\begin{lemma}
\label{lem:limit}
Suppose that $f_n$ is a sequence of integrable weakly bell-shaped extended functions described by Corollary~\ref{cor:bell1} with integral at most one, which correspond to parameters $a_n$, $b_n$, $c_n$ and $\ph_n$ in representation~\eqref{eq:bell1} of $\laplace f_n(i \xi)$. Suppose that for every $\xi \in \R$, $\laplace f_n(i \xi)$ converges to a finite, non-zero limit $\Phi(\xi)$. Then $\Phi(\xi)$ is the Fourier transform of a weakly bell-shaped extended function $f$ with integral at most one, and $\Phi(\xi) = \laplace f(i \xi)$ is given by~\eqref{eq:bell1} for some parameters $a$, $b$, $c$ and $\ph$. More precisely, $b$ is the limit of $b_n$, $\ph(s) ds$ is the vague limit of $\ph_n(s) ds$, and $a$ and $c$ are decribed by the following properties: $a \delta_0(ds) + s \ph(1/s) ds$ is the vague limit of $a_n \delta_0(ds) + s \ph_n(1/s) ds$, while $c \delta_0(ds) + s^{-3} \ph(s) ds$ is the vague limit of $c_n \delta_0(ds) + s^{-3} \ph_n(s) ds$.
\end{lemma}

\begin{proof}
By Remark~3.16 in~\cite{rogers}, pointwise convergence of $f_n$ implies convergence of $b_n$ to some $b$, and weak convergence of the sequence of measures $\min\{s^{-2}, s^{-4}\} s \ph_n(s) ds + \tfrac{1}{\pi} c_n \delta_0(ds) + \tfrac{1}{\pi} a_n \delta_{\infty}(ds)$ on $\R \cup \{\infty\}$ (the one-point compactification of $\R$). Since a more general result is given in Lemma~\ref{lem:limit2} below, we omit the details (note that the expression used in~\cite{rogers} is slightly different than~\eqref{eq:bell1}).

Due to our assumptions, namely: $\ph_n(s_2) \ge \ph_n(s_1) - 1$ when $s_1 \le s_2$, the limit measure mentioned above is necessarily of the form $\min\{s^{-2}, s^{-4}\} s \ph(s) ds + \tfrac{1}{\pi} c \delta_0(ds) + \tfrac{1}{\pi} a \delta_{\infty}(ds)$, where $\ph(s) ds$ is the vague limit of $\ph_n(s) ds$; again we refer to Lemma~\ref{lem:limit2} for further details. In particular, we have $a \ge 0$, $c \ge 0$, and $\ph$ satisfies the integrability condition~\ref{it:bell1:b} in Corollary~\ref{cor:bell}. Furthermore, after modification on a set of Lebesgue measure zero, $\ph$ satisfies the level-crossing condition~\ref{it:bell1:a} in Corollary~\ref{cor:bell1} (which is preserved by the vague limit). Passing to the limit in the expression~\eqref{eq:bell1} for $\laplace f_n(i \xi)$ (that is, with $a$, $b$, $c$ and $\ph$ replaced by $a_n$, $b_n$, $c_n$ and $\ph_n$), we find that the limit is again given by the right-hand side of~\eqref{eq:bell1}, and thus it is the Fourier transform of a weakly bell-shaped integrable function $f$, as desired. Finally, the expressions for $a$, $c$ and $\ph$ follow from the weak convergence of measures discussed above.
\end{proof}

\begin{lemma}
\label{lem:F:inv}
If $f$ is an integrable strictly bell-shaped function and $\xi \in \R \setminus \{0\}$, then
\formula[eq:F:inv]{
 \laplace f(i \xi) & = \lim_{n \to \infty} \biggl(\frac{n^{n + 1}}{n! \, \xi^n} \int_{-\infty}^\infty \frac{f^{(n)}(n x)}{1 + i \xi x} \, dx\biggr) .
}
\end{lemma}

\begin{proof}
For $s > 0$ denote
\formula{
 \Phi(s) & = \int_0^\infty e^{-s \xi} \laplace f(i \xi) d\xi .
}
Note that $\laplace f(i \xi)$ is a continuous and bounded function of $\xi > 0$; thus, $\Phi$ is well-defined. By Fubini's theorem, for $s > 0$ we have
\formula{
 \Phi(s) & = \int_{-\infty}^\infty \frac{f(x)}{s + i x} \, dx .
}
Differentiating under the integral, we find that
\formula{
 \Phi^{(n)}(s) & = (-1)^n n! \int_{-\infty}^\infty \frac{f(x)}{(s + i x)^{n + 1}} \, dx ,
}
and $n$-fold integration by parts leads to
\formula{
 \Phi^{(n)}(s) & = i^n \int_{-\infty}^\infty \frac{f^{(n)}(x)}{s + i x} \, dx .
}
By Post's inversion formula for the Laplace transform, for $\xi > 0$ we have
\formula{
 \laplace f(i \xi) & = \lim_{n \to \infty} \biggl(\frac{(-1)^n n^{n + 1}}{n!} \, \frac{\Phi^{(n)}(n / \xi)}{\xi^{n + 1}}\biggr) .
}
It follows that
\formula{
 \laplace f(i \xi) & = \lim_{n \to \infty} \biggl(\frac{(-i)^n n^{n + 1}}{n! \, \xi^{n + 1}} \int_{-\infty}^\infty \frac{f^{(n)}(x)}{n / \xi + i x} \, dx\biggr) .
}
Substituting $x = n y$, we conclude that
\formula{
 \laplace f(i \xi) & = \lim_{n \to \infty} \biggl(\frac{n^{n + 1}}{n! \, (i \xi)^n} \int_{-\infty}^\infty \frac{f^{(n)}(n y)}{1 + i \xi y} \, dy\biggr) ,
}
as desired. Since $\laplace f(-i \xi) = \overline{\laplace f(i \xi)}$, the above equality also holds when $\xi < 0$, and the proof is complete.
\end{proof}

\begin{proof}[Proof of Theorem~\ref{thm:bell2}, integrable case]
We claim that it is sufficient to prove the theorem for strictly bell-shaped functions. Indeed, suppose that every integrable strictly bell-shaped function is described by Corollary~\ref{cor:bell1}. If an integrable extended function $f$ is weakly bell-shaped, then for every $t > 0$, $f * G_t$ is integrable and strictly bell-shaped, and hence it is described by Corollary~\ref{cor:bell1}. Clearly, the corresponding parameters $a_t - t$, $b_t$, $c_t$ and $\ph_t$ do not depend on $t > 0$. By passing to a limit as $t \to 0^+$ in~\eqref{eq:bell1}, we obtain the desired representation~\eqref{eq:bell1} of $\laplace f(i \xi)$.

Suppose that $f$ is integrable and strictly bell-shaped. With no loss of generality we may assume that the integral of $f$ is one. For $n = 0, 1, 2, \ldots$ we denote the zeroes of $f^{(n)}$ by $n \alpha_{n,k}$, where $k = 1, 2, \ldots, n$. More precisely, we assume that
\formula{
 -\infty & = \alpha_{n,0} < \alpha_{n,1} < \alpha_{n,2} < \ldots < \alpha_{n,n} < \alpha_{n,n+1} = \infty
}
and $(-1)^k f^{(n)}(x) \ge 0$ for $x \in (\alpha_{n,k}, \alpha_{n,k+1})$, $k = 0, 1, 2, \ldots, n$. We define
\formula{
 g_n(x) & = \frac{(-1)^n n^{n + 1}}{n!} \, f^{(n)}(n x) \prod_{k = 1}^n (x - \alpha_{n,k})
}
(a very similar function plays a key role in Hirschman's proof of Schoenberg's conjecture in~\cite{hirschman}). Clearly, $g_n(x) \ge 0$ for all $x \in \R$. Integrating by parts $n$ times and using Lemma~\ref{lem:fn:lim}, we find that
\formula[eq:gn:int]{
 \int_{-\infty}^\infty g_n(x) dx & = \int_{-\infty}^\infty n f(n x) dx = \int_{-\infty}^\infty f(y) dy = 1 ,
}
so that $g_n$ is integrable. We now compute (a variant of) the Cauchy--Stieltjes transform of $g_n$. Observe that for a fixed $z \in \C \setminus \R$, we have
\formula{
 \frac{1}{1 + x z} \prod_{k = 1}^n (x - \alpha_{n,k}) & = \frac{(-1)^n }{z^n (1 + x z)} \prod_{k = 1}^n (1 + \alpha_{n,k} z) + P_n(x)
}
for all $x \in \R$, where $P_n$ is a polynomial of degree at most $n - 1$. Furthermore, the same argument that led us to~\eqref{eq:gn:int} implies that
\formula{
 \int_{-\infty}^\infty f^{(n)}(n x) P_n(x) dx & = 0 .
}
Therefore, for $z \in \C \setminus \R$, we have
\formula{
 \int_{-\infty}^\infty \frac{g_n(x)}{1 + x z} \, dx & = \frac{(-1)^n n^{n + 1}}{n!} \int_{-\infty}^\infty \frac{f^{(n)}(n x)}{1 + x z} \biggl(\prod_{k = 1}^n (x - \alpha_{n,k})\biggr) dx \\
 & = \frac{n^{n + 1}}{n! \, z^n} \biggl(\prod_{k = 1}^n (1 + \alpha_{n,k} z)\biggr) \int_{-\infty}^\infty \frac{f^{(n)}(n x)}{1 + x z} \, dx .
}
Equivalently,
\formula[eq:F:approx]{
 \frac{n^{n + 1}}{n! \, z^n} \int_{-\infty}^\infty \frac{f^{(n)}(n x)}{1 + x z} \, dx & = \biggl(\prod_{k = 1}^n \frac{1}{1 + \alpha_{n,k} z}\biggr) \int_{-\infty}^\infty \frac{g_n(x)}{1 + x z} \, dx .
}
By Lemma~\ref{lem:F:inv}, for $z = i \xi$ with $\xi \in \R \setminus \{0\}$, the left-hand side of~\eqref{eq:F:approx} converges as $n \to \infty$ to $\laplace f(i \xi)$. Let us inspect the right-hand side.

The factor $\prod_{k = 1}^n (1 + i \alpha_{n,k} \xi)^{-1}$, if not identically equal to one, is the Fourier transform of a Pólya frequency function; see~\eqref{eq:pff}, or Proposition~5.3 in~\cite{bell}. The other factor in the right-hand side of~\eqref{eq:F:approx} is the Fourier transform of an integrable $\amcm$ function, with integral equal to one. Indeed, by Fubini's theorem, for $\xi \in \R \setminus \{0\}$ we have
\formula{
 \int_{-\infty}^\infty \frac{g_n(x)}{1 + i \xi x} \, dx & = \int_0^\infty \frac{g_n(-x)}{x (1/x - i \xi)} \, dx + \int_0^\infty \frac{g_n(x)}{x (1/x + i \xi)} \, dx \\
 & = \int_{-\infty}^0 \biggl(\int_0^\infty \frac{g_n(-x) e^{s/x}}{x} \, dx \biggr) e^{-i \xi s} ds + \int_0^\infty \biggl(\int_0^\infty \frac{g_n(x) e^{-s/x}}{x} \, dx \biggr) e^{-i \xi s} ds .
}
The first parenthesised expression in the right-hand side is an absolutely monotone function of $x \in (-\infty, 0)$, while the other one is a completely monotone function of $x \in (0, \infty)$ (this argument is very similar to the derivation of~\eqref{eq:amcm:int}; see also Proposition~5.1 in~\cite{bell}). Furthermore, the integral of the $\amcm$ function in question is equal to the value of its Fourier transform at zero. This value is given by the integral of $g_n$, which, as we already know, is equal to one.

Let us summarise our findings. We have proved that for $\xi \in \R \setminus \{0\}$,
\formula[eq:F]{
 \laplace f(i \xi) & = \lim_{n \to \infty} \biggl(\prod_{k = 1}^n \frac{1}{1 + i \alpha_{n,k} \xi}\biggr) \int_{-\infty}^\infty \frac{g_n(x)}{1 + i \xi x} \, dx ,
}
and that for every $n = 0, 1, 2, \ldots$ the expression under the limit in the right-hand side is the product of Fourier transforms of an integrable $\amcm$ function and a Pólya frequency function. As explained in the introduction (see Corollary~\ref{cor:bell}), this means that the expressions under the limit in~\eqref{eq:F} are given by the right-hand side of~\eqref{eq:bell1} for appropriate parameters $a_n$, $b_n$, $c_n$ and $\ph_n$, accompanied by conditions~\ref{it:bell1:a} and \ref{it:bell1:b} in Corollary~\ref{cor:bell1}. By Lemma~\ref{lem:limit}, also the Fourier transform of $f$ is given by~\eqref{eq:bell1}, as desired.
\end{proof}

%
%                            ---------- o ----------
%

\section{Bell-shaped functions, non-integrable case}
\label{sec:bs2}

Here we provide a complete proof of Theorem~\ref{thm:bell2} in the general case, without additional integrability condition. Compared to Section~\ref{sec:bs1}, there are two essential difficulties, and the proofs of Lemmas~\ref{lem:limit} and~\ref{lem:F:inv} need to be modified appropriately.

If $f$ is not integrable, the Fourier transform of $f$ can be defined in a number of equivalent ways. We follow the most elementary definition in terms of an improper integral, already mentioned in Theorem~\ref{thm:bell}. We always assume that an extended function $f$ is locally integrable, converges to zero at $\pm \infty$, and it is monotone near $-\infty$ and $\infty$. We note that if $f$ is monotone on $(-\infty, -p]$ and on $[p, \infty)$, then we can re-write the expression for the Fourier transform of $f$ in terms of absolutely convergent improper Riemann--Stieltjes integrals:
\formula[eq:lf]{
\begin{aligned}
 \laplace f(i \xi) & = \int_{-\infty}^{-p} e^{-i \xi x} f(x) dx + \int_{-p}^p e^{-i \xi x} f(x) dx + \int_p^\infty e^{-i \xi x} f(x) dx \\
 & = \biggl(-\frac{e^{i \xi p} f(-p)}{i \xi} + \int_{-\infty}^{-p} \frac{e^{-i \xi x}}{i \xi} \, df(x) \biggr) + \int_{-p}^p e^{-i \xi x} f(x) dx \\
 & \hspace*{14em} + \biggl( \frac{e^{-i \xi p} f(p)}{i \xi} + \int_p^\infty \frac{e^{-i \xi x}}{i \xi} \, df(x) \biggr)
\end{aligned}
}
(we integrated by parts in the integrals over $(-\infty, -p)$ and $(p, \infty)$). If $f$ is strictly bell-shaped, then $f'$, $f''$ etc.\@ are absolutely integrable, and we may simplify the above expression to an absolutely convergent integral
\formula[eq:lf1]{
 \laplace f(i \xi) & = \int_{-\infty}^\infty \frac{e^{-i \xi x}}{i \xi} \, f'(x) dx = \int_{-\infty}^\infty \frac{e^{-i \xi x}}{(i \xi)^2} \, f''(x) dx = \ldots
}
Clearly, in the above formula we may replace $e^{-i \xi x}$ by $e^{-i \xi x} - 1$.

We first extend Lemma~\ref{lem:limit}, by showing that the class of functions given by~\eqref{eq:bell} is closed under pointwise limits.

\begin{lemma}
\label{lem:limit2}
Suppose that $f_n$ is a sequence of weakly bell-shaped extended functions described by Theorem~\ref{thm:bell}, which correspond to parameters $a_n$, $b_n$, $c_n$ and $\ph_n$ in representation~\eqref{eq:bell} of $\laplace f_n(i \xi)$, the Fourier transform of $f_n$. Suppose that for every $\xi \in i \R \setminus \{0\}$, $\laplace f_n(i \xi)$ converges to a finite limit $\Phi(\xi)$, and that $\Phi$ is not identically zero. Then $\Phi(\xi)$ is given by~\eqref{eq:bell} for some $a$, $b$, $c$ and $\ph$ which satisfy conditions~\ref{thm:bell:a} and~\ref{thm:bell:b} of Theorem~\ref{thm:bell}. Furthermore, $b$ and $c$ are the limits of $b_n$ and $c_n$, $\ph(s) ds$ is the vague limit of $\ph_n(s) ds$, and $a$ is determined by the following property: $a \delta_0(ds) + s \ph(1 / s) ds$ is the vague limit of $a_n \delta_0(ds) + s \ph_n(1 / s) ds$.
\end{lemma}

\begin{proof}
As in~\eqref{eq:bell:rlog}, we have
\formula[eq:bell:rlog2]{
 \log |\laplace f_n(i \xi)| & = -a_n \xi^2 + c_n + \int_{-\infty}^\infty \biggl( \frac{s}{\xi^2 + s^2} - \frac{1}{s} \, \ind_{\R \setminus (-1, 1)}(s) \biggr) \ph_n(s) ds .
}
Thus, setting $\xi = \xi_1$ and $\xi = \xi_2$ for $\xi_1, \xi_2 \in \R \setminus \{0\}$, $\xi_1 \ne \xi_2$, in the above formula, we easily find that
\formula{
 \frac{\log |\laplace f_n(i \xi_1)| - \log |\laplace f_n(i \xi_2)|}{\xi_2^2 - \xi_1^2} & = a_n + \int_{-\infty}^\infty \frac{s \ph_n(s)}{(\xi_1^2 + s^2) (\xi_2^2 + s^2)} \, ds .
}
Note that $a_n \ge 0$ and $s \ph_n(s) \ge 0$. By assumption, $f_n(i \xi_1)$ and $f_n(i \xi_2)$ converge to a finite limit, and we may assume that the limit of $f_n(i \xi_2)$ is non-zero. It follows that there is a number $C_1$ such that
\formula[eq:limit:est]{
 a_n + \int_{-\infty}^\infty \min\{1, s^{-4}\} s \ph_n(s) ds & \le C_1
}
for every $n = 1, 2, \ldots$\, Before we proceed, we need to slightly improve~\eqref{eq:limit:est}. The level-crossing condition~\ref{thm:bell:a} in Theorem~\ref{thm:bell} asserts that $\ph_n(s_2) - \ph_n(s_1) \ge -1$ when $s_2 \ge s_1$. Thus, for $s > 0$ we have
\formula{
 \ph_n(s) & \le \int_0^1 (1 + \ph(s + t)) dt \\
 & \le 1 + (s + 1)^3 \int_0^1 \ph_n(s + t) \min\{1, (s + t)^{-3}\} \, ds \le 1 + (s + 1)^3 C_1 .
}
A similar estimate holds for $s < 0$, and hence for all $s \in \R$ we have
\formula[eq:limit:bound]{
 |\ph_n(s)| & \le 1 + (|s| + 1)^3 C_1 .
}
Using the above bound and~\eqref{eq:limit:est}, we find that
\formula[eq:limit:est2]{
 a_n + \int_{-\infty}^\infty \min\{|s|^{-1}, s^{-4}\} s \ph_n(s) ds & \le \int_{-1}^1 |\ph_n(s)| ds + C_1 \le C_2 ,
}
where $C_2 = 2 (1 + \tfrac{15}{4} C_1) + C_1$. We have thus proved that the sequence of measures
\formula[eq:limit:mu]{
 \mu_n(ds) & = \min\{|s|^{-1}, s^{-4}\} s \ph_n(s) ds + a_n \delta_{\infty}(ds)
}
is relatively compact with respect to the topology of weak convergence on $\R \cup \{\infty\}$, the one-point compactification of $\R$. Later we will prove that in fact $\mu_n$ has a weak limit.

Suppose that $\mu$ is a partial limit of $\mu_n$; that is, $\mu$ is the weak limit on $\R \cup \{\infty\}$ of some sub-sequence $\mu_{j(n)}$. Define $a = \mu(\{\infty\})$; clearly, $a \ge 0$. By~\eqref{eq:limit:bound} and~\eqref{eq:limit:mu}, the density functions of $\mu_n(ds)$ on $\R$ are non-negative and bounded by a constant $C_3$ uniformly in $s \in \R$ and $n = 1, 2, \ldots$\, Therefore, $\mu$ is absolutely continuous on $\R$ with respect to the Lebesgue measure, with a density function bounded by $C_3$ almost everywhere. If we denote the density function of $\mu(ds)$ by $\min\{|s|^{-1}, s^{-4}\} s \ph(s)$, then it follows that $\ph(s) \ge 0$ for almost all $s > 0$, $\ph(s) \le 0$ for almost all $s < 0$,
\formula{
 \mu(ds) & = \min\{|s|^{-1}, s^{-4}\} s \ph(s) ds + a \delta_{\infty}(ds) ,
}
and $a$ and $\ph$ satisfy analogues of~\eqref{eq:limit:bound} (almost everywhere) and~\eqref{eq:limit:est2}. For $\xi \in \R \setminus \{0\}$ we define
\formula{
 \Psi_n(\xi) & = \exp(i b_n \xi - c_n) \laplace f_n(i \xi) .
}
By the representation~\eqref{eq:bell} of $\laplace f_n(i \xi)$, we have
\formula{
 \Psi_n(\xi) & = \exp\biggl(-a_n \xi^2 + \int_{-\infty}^\infty \biggl( \frac{1}{i \xi + s} - \biggl(\frac{1}{s} - \frac{i \xi}{s^2} \biggr) \ind_{\R \setminus (-1, 1)}(s) \biggr) \ph_n(s) ds\biggr) \\
 & = \exp\biggl(-\int_{\R \cup \{\infty\}} g_\xi(s) \mu_n(ds) \biggr) ,
}
where, by a short calculation,
\formula{
 g_\xi(s) & = -\frac{1}{i \xi + s} \, \sign s \, \ind_{(-1, 1)}(s) - \biggl( \frac{1}{i \xi + s} - \frac{1}{s} + \frac{i \xi}{s^2} \biggr) s^3 \ind_{\R \setminus (-1, 1)}(s) + \xi^2 \ind_{\{\infty\}}(s) \\
 & = \frac{-\sign s}{i \xi + s} \, \ind_{(-1, 1)}(s) + \frac{\xi^2 s^2}{s^2 (i \xi + s)} \, \ind_{\R \setminus (-1, 1)}(s) + \xi^2 \ind_{\{\infty\}}(s) .
}
In particular, $g_\xi$ is a bounded function on $\R \cup \{\infty\}$, continuous except possibly at $-1$, $0$ and $1$. Since the limiting measure $\mu$ does not charge $\{-1, 0, 1\}$, we have
\formula{
 \lim_{n \to \infty} \Psi_{j(n)}(\xi) & = \exp\biggl(-\lim_{n \to \infty} \int_{\R \cup \{\infty\}} g_\xi(s) \mu_{j(n)}(ds) \biggr) \\
 & = \exp\biggl(-\int_{\R \cup \{\infty\}} g_\xi(s) \mu(ds) \biggr) \\
 & = \exp\biggl(-a \xi^2 + \int_{-\infty}^\infty \biggl( \frac{1}{i \xi + s} - \biggl(\frac{1}{s} - \frac{i \xi}{s^2} \biggr) \ind_{\R \setminus (-1, 1)}(s) \biggr) \ph(s) ds\biggr) .
}
On the other hand,
\formula{
 \lim_{n \to \infty} \bigl(\exp(-i b_{j(n)} \xi + c_{j(n)}) \Psi_{j(n)}(\xi)\bigr) & = \lim_{n \to \infty} \laplace f_n(i \xi) = \Phi(\xi) .
}
In particular, $\exp(-i b_{j(n)} \xi + c_{j(n)})$ has a limit for every $\xi \in \R \setminus \{0\}$, and this limit is not everywhere zero. It is now easy to see that $b_{j(n)}$ and $c_{j(n)}$ necessarily converge to some $b, c \in \R$. We conclude that
\formula{
 \Phi(\xi) & = \exp\biggl(-a \xi^2 - i b \xi + c + \int_{-\infty}^\infty \biggl( \frac{1}{i \xi + s} - \biggl(\frac{1}{s} - \frac{i \xi}{s^2} \biggr) \ind_{\R \setminus (-1, 1)}(s) \biggr) \ph(s) ds\biggr)
}
for $\xi \in \R \setminus \{0\}$, that is, $\Phi$ is indeed given by~\eqref{eq:bell}, with $\ph$ satisfying the integrability condition~\ref{thm:bell:b} in Theorem~\ref{thm:bell}, as well as the level-crossing condition~\ref{thm:bell:a} for $k = 0$.

Since for every $n = 1, 2, \ldots$, $\ph_n$ satisfies the level-crossing condition~\ref{thm:bell:a}, for every $k \in \Z$ the measure $(\ph_{j(n)}(s) - k) ds$ changes its sign at most once. Furthermore, the measures $(\ph_{j(n)}(s) - k) ds$ converge vaguely on $\R$ to $(\ph(s) - k) ds$. Therefore, the latter measure also changes its sign at most once. We conclude that, after modification on a set of zero Lebesgue measure, $\ph$ satisfies the level-crossing condition~\ref{thm:bell:a} of Theorem~\ref{thm:bell} for every $k \in \Z$. The first assertion of the lemma is therefore proved.

The parameters $a$, $b$, $c$ and $\ph$ are determined uniquely by the values of $\Phi$ on $\R \setminus \{0\}$; see Remark~5.5 in~\cite{bell} for further discussion. Therefore, every partial limit of the sequence of measures $\mu_n$ (in the sense of weak convergence on $\R \cup \{\infty\}$) is necessarily equal to the measure $\mu$ described above, and in addition the corresponding partial limits of $b_n$ and $c_n$ exist and are equal to $b$ and $c$, respectively. This proves that $\mu_n$ converges weakly to $\mu$, and also $b_n$ and $c_n$ converge to $b$ and $c$, respectively. By the definitions of $\mu_n$ and $\mu$ and a substitution $r = 1 / s$, this is equivalent to the second assertion of the lemma.
\end{proof}

The following result extends Lemma~\ref{lem:F:inv} to non-integrable bell-shaped functions.

\begin{lemma}
\label{lem:F:inv2}
If $f$ is a strictly bell-shaped function and $\xi \in \R \setminus \{0\}$, then
\formula[eq:F:inv2]{
 \laplace f(i \xi) & = \lim_{n \to \infty} \biggl(\frac{n^{n + 1}}{n! \, \xi^n} \int_{-\infty}^\infty \frac{f^{(n)}(n x)}{1 + i \xi x} \, dx\biggr) .
}
\end{lemma}

\begin{proof}
The proof is very similar, except that we need to handle a singularity of $\laplace f(i \xi)$ at $\xi = 0$. Observe that for $\xi > 0$, $(1 + \xi)^{-1} \xi |\laplace f(i \xi)|$ is a bounded function. Indeed, if $f$ is non-decreasing on $(-\infty, -p]$, non-increasing on $[p, \infty)$, and bounded by $M$, then, by~\eqref{eq:lf},
\formula{
 |\laplace f(i \xi)| & \le \frac{|f(-p)|}{\xi} + \int_{-\infty}^{-p} \frac{1}{\xi} \, df(x) + 2 p M + \frac{|f(p)|}{\xi} + \int_p^\infty \frac{1}{\xi} d(-f)(x) \le \frac{4 M}{\xi} + 2 p M .
}
For $s > 0$ we denote
\formula{
 \Psi(s) & = \int_0^\infty e^{-s \xi} \xi \laplace f(i \xi) d\xi .
}
Note that in the proof of Lemma~\ref{lem:F:inv}, for integrable $f$, we defined $\Phi$ to be the Laplace transform of $f(i \xi)$; in this case, $\Psi(s) = -\Phi'(s)$.

We claim that for $s > 0$ we have
\formula[eq:psi2]{
 \Psi(s) & = \int_{-\infty}^\infty \frac{f(x)}{(s + i x)^2} \, dx .
}
This is a consequence of~\eqref{eq:lf} and Fubini's theorem: we write
\formula{
 \Psi(s) & = \int_0^\infty e^{-s \xi} \xi \biggl( -\frac{e^{i \xi p} f(-p)}{i \xi} + \int_{-\infty}^{-p} \frac{e^{-i \xi x}}{i \xi} \, df(x) + \int_{-p}^p e^{-i \xi x} f(x) dx \\
 & \hspace*{17em} + \frac{e^{-i \xi p} f(p)}{i \xi} + \int_p^\infty \frac{e^{-i \xi x}}{i \xi} \, df(x)\biggr) d\xi \\
 & = -\frac{f(-p)}{i (s - i p)} + \int_{-\infty}^{-p} \frac{1}{i (s + i x)} \, df(x) + \int_{-p}^p \frac{f(x)}{(s + i x)^2} \, dx \\
 & \hspace*{17em} + \frac{f(p)}{i (s + i p)} + \int_p^\infty \frac{1}{i (s + i x)} \, df(x) ,
}
and integration by parts leads us back to~\eqref{eq:psi2}.

The next part of the proof is very similar to the proof of Lemma~\ref{lem:F:inv}. Differentiating under the integral, we obtain
\formula{
 \Psi^{(n-1)}(s) & = (-1)^{n - 1} n! \int_{-\infty}^\infty \frac{f(x)}{(s + i x)^{n + 1}} \, dx ,
}
and $n$-fold integration by parts leads to
\formula[eq:dpsi]{
 \Psi^{(n-1)}(s) & = -i^n \int_{-\infty}^\infty \frac{f^{(n)}(x)}{s + i x} \, dx .
}
Post's inversion formula for the Laplace transform tells us that for $\xi > 0$ we have
\formula{
 \xi \laplace f(i \xi) & = \lim_{n \to \infty} \biggl(\frac{(-1)^{n - 1} (n - 1)^n}{(n - 1)!} \, \frac{\Psi^{(n - 1)}((n - 1) / \xi)}{\xi^n}\biggr) .
}
Here we need a minor modification. Since $\xi \laplace f(i \xi)$ is a continuous function of $\xi > 0$, the expressions under the limit in the right-hand side converge locally uniformly in $\xi > 0$. Thus, we may replace $\xi$ by $n \xi / (n - 1)$. This leads to the expression
\formula{
 \xi \laplace f(i \xi) & = \lim_{n \to \infty} \biggl(\frac{(-1)^{n - 1} n^n}{(n - 1)!} \, \frac{\Psi^{(n - 1)}(n / \xi)}{\xi^n}\biggr) ,
}
or, in other words,
\formula{
 \laplace f(i \xi) & = \lim_{n \to \infty} \biggl(\frac{(-i)^n n^n}{(n - 1)! \, \xi^{n + 1}} \int_{-\infty}^\infty \frac{f^{(n)}(x)}{n / \xi + i x} \, dx\biggr) .
}
Substituting $x = n y$, we conclude that
\formula{
 \laplace f(i \xi) & = \lim_{n \to \infty} \biggl(\frac{n^n}{(n - 1)! \, (i \xi)^n} \int_{-\infty}^\infty \frac{f^{(n)}(n y)}{1 + i \xi y} \, dy\biggr) ,
}
which is equivalent to the assertion of the lemma when $\xi > 0$. The case $\xi < 0$ also follows, because $\laplace f(-i \xi) = \overline{\laplace f(i \xi)}$.
\end{proof}

\begin{proof}[Proof of Theorem~\ref{thm:bell2}]
As in the proof for integrable bell-shaped functions, it is sufficient to prove the theorem for strictly bell-shaped functions. Therefore, we suppose that $f$ is a strictly bell-shaped function. As in Section~\ref{sec:bs1}, for $n = 0, 1, 2, \ldots$ we denote the zeroes of $f^{(n)}$ by $n \alpha_{n,k}$, where $k = 1, 2, \ldots, n$, and we let
\formula[eq:gn2]{
 g_n(x) & = \frac{(-1)^n n^{n + 1}}{n!} \, f^{(n)}(n x) \prod_{k = 1}^n (x - \alpha_{n,k}) .
}
Again, $g_n(x) \ge 0$ for all $x \in \R$, and by Lemma~\ref{lem:fn:lim}, $g_n$ converges to zero at $\pm \infty$.

Our goal is again to express (a variant of) the Cauchy--Stieltjes transform of $f^{(n)}$ in terms of a similar transform of $g_n$. This time, however, $g_n$ need not be integrable. Nevertheless, we will show that the Cauchy--Stieltjes transform of $g_n$ is well-defined.

As in the proof in the integrable case, for a fixed $z \in \C \setminus \R$, we have
\formula[eq:pf2]{
 \frac{1}{1 + x z} \prod_{k = 1}^n (x - \alpha_{n,k}) & = \frac{(-1)^n}{z^n (1 + x z)} \prod_{k = 1}^n (1 + \alpha_{n,k} z) + P_n(x)
}
for all $x \in \R$, where $P_n$ is a polynomial of degree at most $n - 1$. Integrating by parts $n$ times and using Lemma~\ref{lem:fn:lim}, we find that, as an improper integral,
\formula[eq:fpn2]{
 \int_{-\infty}^\infty f^{(n)}(n x) P_n(x) dx & = 0
}
(we will momentarily see that in fact the above integral converges absolutely).

Suppose that $n \ge 1$, so that $f^{(n)}$ is absolutely integrable. In this case, by~\eqref{eq:pf2} and~\eqref{eq:fpn2},
\formula[eq:cgn2]{
\begin{aligned}
 \int_{-\infty}^\infty \frac{g_n(x)}{1 + x z} \, dx & = \frac{(-1)^n n^{n + 1}}{n!} \int_{-\infty}^\infty \frac{f^{(n)}(n x)}{1 + x z} \biggl(\prod_{k = 1}^n (x - \alpha_{n,k})\biggr) dx \\
 & = \frac{n^{n + 1}}{n! \, z^n} \biggl(\prod_{k = 1}^n (1 + \alpha_{n,k} z)\biggr) \int_{-\infty}^\infty \frac{f^{(n)}(n x)}{1 + x z} \, dx .
\end{aligned}
}
Here the integral in the right-hand side converges absolutely, the improper integral in~\eqref{eq:fpn2} converges, and thus the left-hand side is well-defined as an improper integral. However, $g_n(x) \ge 0$ and $\im (1 + x z)^{-1} = x |1 + x z|^{-2} \im z$ has constant sign, and $|\im (1 + x z)^{-1}|$ is comparable with $|1 + x z|^{-1}$ as $x \to \pm \infty$. Therefore, convergence of the improper integral of $g_n(x) / (1 + x z)$ automatically implies absolute convergence of this integral. It follows that, just as in the integrable case,
\formula[eq:cf2]{
 \frac{n^{n + 1}}{n! \, z^n} \int_{-\infty}^\infty \frac{f^{(n)}(n x)}{1 + x z} \, dx & = \biggl(\prod_{k = 1}^n \frac{1}{1 + \alpha_{n,k} z}\biggr) \int_{-\infty}^\infty \frac{g_n(x)}{1 + x z} \, dx ,
}
with both integrals absolutely convergent. Furthermore, by Lemma~\ref{lem:F:inv2}, for $z = i \xi$ with $\xi \in \R \setminus \{0\}$, the left-hand side converges as $n \to \infty$ to $\laplace f(i \xi)$.

We have thus proved that for $\xi \in \R \setminus \{0\}$,
\formula[eq:F2]{
 \laplace f(i \xi) & = \lim_{n \to \infty} \biggl(\prod_{k = 1}^n \frac{1}{1 + i \alpha_{n,k} \xi}\biggr) \int_{-\infty}^\infty \frac{g_n(x)}{1 + i \xi x} \, dx .
}
For every $n = 0, 1, 2, \ldots$ the expression under the limit in the right-hand side is the product of Fourier transforms of a locally integrable $\amcm$ function (by Proposition~\ref{prop:amcm}) and a Pólya frequency function (by Proposition~\ref{prop:pff}); that is, it is the Fourier transform of a bell-shaped function described by Theorem~\ref{thm:bell}. By Lemma~\ref{lem:limit2}, $\laplace f(i \xi)$ is given by~\eqref{eq:bell}, with $a \ge 0$, $b \in \R$, $c \in \R$ and $\ph$ satisfying the level-crossing condition~\ref{thm:bell:a} and the integrability condition~\ref{thm:bell:b} of Theorem~\ref{thm:bell}. It remains to observe that condition~\ref{thm:bell:c} of this theorem follows directly from the assumptions. Indeed, by~\eqref{eq:lf1} (with $e^{-i \xi x} - 1$ in the numerator), we have
\formula{
 \xi \im \laplace f(i \xi) & = \im \int_{-\infty}^\infty \frac{e^{-i \xi x} - 1}{i} \, f'(x) dx = \int_{-\infty}^\infty (\cos(\xi x) - 1) f'(x) dx ,
}
and the right-hand side converges to zero as $\xi \to 0$ by the dominated convergence theorem. Similarly, again by~\eqref{eq:lf1},
\formula{
 \re \laplace f(i \xi) & = \re \int_{-\infty}^\infty \frac{e^{-i \xi x} - 1}{(i \xi)^2} \, f''(x) dx = \int_{-\infty}^\infty \frac{1 - \cos(\xi x)}{\xi^2} \, f''(x) dx ,
}
and the right-hand side is integrable with respect to $\xi \in (-1, 1)$.
\end{proof}

\begin{remark}
\label{rem:aux}
As a supplement to the above proof, we make the following observation that will be needed in the next section. Suppose that $f$ is a strictly bell-shaped function. By Theorem~\ref{thm:bell2}, $f$ is described by Theorem~\ref{thm:bell}, that is, the Fourier transform $\laplace f(i \xi)$ is given by~\eqref{eq:bell} for some $a$, $b$, $c$ and $\ph$. In the above proof we expressed $\laplace f(i \xi)$ as a limit of 
\formula{
 & \biggl(\prod_{k = 1}^n \frac{1}{1 + i \alpha_{n,k} \xi}\biggr) \int_{-\infty}^\infty \frac{g_n(x)}{1 + i \xi x} \, dx ,
}
see~\eqref{eq:F2}, and we observed that for every $n = 1, 2, \ldots$ the above expression is again given by the right-hand side of~\eqref{eq:bell} for some parameters $a_n$, $b_n$, $c_n$ and $\ph_n$. It is clear that $a_n = 0$. Thus, as a consequence of Lemma~\ref{lem:limit2} and~\eqref{eq:F2}, we have
\formula[eq:aux:lim]{
 a \delta_0(ds) + s \ph(1 / s) ds & = \lim_{n \to \infty} (s \ph_n(1 / s) ds) ,
}
with the vague limit of measures in the right-hand side. Furthermore, as discussed in the introduction, if
\formula[eq:aux:ph]{
 \ph_{n,h}(s) & = \sum_{k = 1}^n \ind_{(0, \infty)}(\alpha_{n,k}) \ind_{[1/\alpha_{n,k}, \infty)}(s) - \sum_{k = 1}^n \ind_{(-\infty, 0)}(\alpha_{n,k}) \ind_{(-\infty, 1/\alpha_{n,k}]}(s)
}
is the function $\ph$ that corresponds (as in Proposition~\ref{prop:pff}) to the Pólya frequency function with Fourier transform $\prod_{k = 1}^n (1 + i \alpha_{n,k} \xi)^{-1}$, then we have
\formula[eq:aux:est]{
\begin{aligned}
 \ph_{n,h}(s) \le \ph_n(s) & \le \ph_{n,h}(s) + 1 & \text{for almost all $s > 0$,} \\
 \ph_{n,h}(s) - 1 \le \ph_n(s) & \le \ph_{n,h}(s) & \text{for almost all $s < 0$.}
\end{aligned}
}
\end{remark}

%
%                            ---------- o ----------
%

\section{Further results on bell-shaped functions}
\label{sec:bs3}

In this short section we prove additional results stated in the Introduction. 

\begin{proof}[Proof of Corollary~\ref{cor:roots}]
Suppose that $f$ is a weakly bell-shaped function described by Theorem~\ref{thm:bell} with parameters $a$, $b$, $c$ and $\ph$, and $n = 1, 2, \ldots$\, Let $g$ be the weakly bell-shaped function as in Theorem~\ref{thm:bell}, with parameters $a/n$, $b/n$, $c/n$ and $\ph/n$. It is straightforward to verify that these parameters indeed satisfy conditions~\ref{thm:bell:a}, \ref{thm:bell:a} and~\ref{thm:bell:c}, and that the $n$-fold convolution of $g$ is equal to $f$.
\end{proof}

\begin{proof}[Proof of Corollary~\ref{cor:walk}]
Again, suppose that $f$ is a weakly bell-shaped function described by Theorem~\ref{thm:bell} with parameters $a$, $b$, $c$ and $\ph$. For $n = 1, 2, \ldots$, the $n$-fold convolution $f_n$ of $f$ has Fourier transform given by~\eqref{eq:bell}, with parameters $n a$, $n b$, $n c$ and $n \ph$. Observe that $f_n$ is weakly bell-shaped if and only if $n \ph$ (after modification on a set of zero Lebesgue measure) satisfies the level-crossing condition~\ref{thm:bell:a}. Thus, $f_n$ is weakly bell-shaped for every $n = 1, 2, \ldots$ if and only if (again after modification on a set of zero Lebesgue measure) $\ph$ is a non-decreasing function.
\end{proof}

Corollary~\ref{cor:zeroes} is an immediate consequence of Proposition~\ref{prop:zeroes}, proved below. Corollary~\ref{cor:bell} does not require proof, and Corollary~\ref{cor:der} was already proved in the introduction.

\begin{proof}[Proof of Proposition~\ref{prop:zeroes}]
We will use the observation made in Remark~\ref{rem:aux}. Suppose that $n \alpha_{n,k}$ are the zeroes of $f^{(n)}$, the Fourier transform of $f$ is given by~\eqref{eq:bell}, and $\ph_n$ and $\ph_{n,h}$ are defined as in Remark~\ref{rem:aux}. In particular, $\ph_{n,h}$ is an integer-valued function which has unit jumps at $1 / \alpha_{n,k}$ (as long as $\alpha_{n,k} \ne 0$) for $k = 1, 2, \ldots, n$ (see~\eqref{eq:aux:ph}), we have
\formula[eq:aux:est2]{
 0 \le s \ph_{n,h}(s) & \le s \ph_n(s) \le s \ph_{n,h}(s) + |s|
}
for almost all $s \in \R$ (see~\eqref{eq:aux:est}), and $s \ph_n(1 / s) ds$ converges vaguely to $a \delta_0(ds) + s \ph(1 / s) ds$ (see~\eqref{eq:aux:lim}).

By~\eqref{eq:aux:est2}, the sequence of measures $\mu_n(ds) = s \ph_{n,h}(1 / s) ds$ is relatively compact with respect to the topology of vague convergence. Suppose that $\mu$ is a partial limit of this sequence, a vague limit of a sub-sequence $\mu_{j(n)}$. Again by~\eqref{eq:aux:est2},
\formula{
 a \delta_0 + \max\{0, s \psi(1 / s) - |s|\} ds \le \mu(ds) & \le a \delta_0(ds) + s \ph(1 / s) ds .
}
Therefore, $\mu(\{0\}) = a$, and $\mu$ is absolutely continuous on $\R \setminus \{0\}$ with respect to the Lebesgue measure. Define $\ph_h$ so that $s \ph_h(1 / s)$ is the density function of $\mu$ on $\R \setminus \{0\}$. Then it follows that $\ph_h(s) ds$ is the vague limit of $\ph_{n,h}(s) ds$ on $\R \setminus \{0\}$. It is now easy to see that, after modification on a set of zero Lebesgue measure, $\ph_h$ is non-decreasing, only takes integer values, and satisfies $\ph_h(0) = 0$. In other words, there is a sequence $s_k \in [-\infty, \infty]$, $k \in \Z$, such that $s_0 = 0$, $s_k$ is non-decreasing, and
\formula{
 \ph_h(s) & = \sum_{k = 1}^\infty \ind_{[s_k, \infty)}(s) - \sum_{k = -\infty}^{-1} \ind_{(-\infty, s_k]}(s) .
}
Furthermore, since $\ph$ satisfies the integrability condition~\ref{thm:bell:b} in Theorem~\ref{thm:bell}, we have $s_1 > 0$ and $s_{-1} < 0$ (see the proof of Lemma~5.4 in~\cite{bell}). We let $\alpha_k = 1 / s_k$ for $k \ne 0$ (with the convention that $1 / \pm \infty = 0$), and we claim that in the sense of vague limit of measures,
\formula{
 \lim_{n \to \infty} \sum_{k = 1}^{j(n)} \alpha_{j(n),k}^2 \delta_{\alpha_{j(n),k}}(ds) & = 2 a \delta_0(ds) + \sum_{k \in \Z \setminus \{0\}} \alpha_k^2 \delta_{\alpha_k}(ds) .
}
The above claim is easily seen to be equivalent to the assertion of the proposition.

Suppose that $u$ is a smooth, non-negative, compactly supported function on $\R$, and define $v(s) = s^2 u(s)$ and $w(s) = v'(s) / s$; we extend $w$ continuously, so that $w(0) = v''(0) = 2 u(0)$. Since $\mu$ is the vague limit of a subsequence $\mu_{j(n)}$ and $w$ is continuous and compactly supported, we have
\formula[eq:aux:wlim]{
 \lim_{n \to \infty} \int_\R w(s) \mu_{j(n)}(ds) & = \int_\R w(s) \mu(ds) .
}
Using the definitions of $\mu_n$ and $\ph_{n,h}$, we evaluate the left-hand side of~\eqref{eq:aux:wlim} (with $j(n)$ replaced by $n$):
\formula{
 \int_\R w(s) \mu_n(ds) & = \int_{-\infty}^\infty v'(s) \ph_{n,h}(1 / s) ds \\
 & = \sum_{k = 1}^n \ind_{(0, \infty)}(\alpha_{n,k}) \int_0^{\alpha_{n,k}} v'(s) ds - \sum_{k = 1}^n \ind_{(-\infty, 0)}(\alpha_{n,k}) \int_{\alpha_{n,k}}^0 v'(s) ds \\
 & = \sum_{k = 1}^n (v(\alpha_{n,k}) - v(0)) = \sum_{k = 1}^n \alpha_{n,k}^2 u(\alpha_{n,k}) .
}
The right-hand side of~\eqref{eq:aux:wlim} is evaluated in a similar way, using the properties of $\mu$ established above:
\formula{
 \int_\R w(s) \mu(ds) & = a w(0) + \int_{-\infty}^\infty v'(s) \ph_h(1 / s) ds \\
 & = 2 a u(0) + \sum_{k = 1}^\infty \int_0^{\alpha_k} v'(s) ds - \sum_{k = -\infty}^{-1} \int_{\alpha_k}^0 v'(s) ds \\
 & = 2 a u(0) + \sum_{k \in \Z \setminus \{0\}} (v(\alpha_k) - v(0)) = 2 a u(0) + \sum_{k \in \Z \setminus \{0\}} \alpha_k^2 u(\alpha_k)
}
(in particular, the sum in the right-hand side is absolutely convergent). We have thus proved that
\formula{
 \lim_{n \to \infty} \sum_{k = 1}^{j(n)} \alpha_{j(n),k}^2 u(\alpha_{j(n),k}) & = 2 a u(0) + \sum_{k \in \Z \setminus \{0\}} \alpha_k^2 u(\alpha_k)
}
whenever $u$ is smooth, non-negative and compactly supported. By approximation, the above equality also holds for general continuous and compactly supported $u$, which is precisely our claim.
\end{proof}

%
%                            ---------- o ----------
%

\section{Whale-shaped functions}
\label{sec:ws}

We conclude the article with a sketch of the proof of our result on whale-shaped functions, Theorem~\ref{thm:whale}.

\begin{proof}[Sketch of the proof of Theorem~\ref{thm:whale}]
Suppose that $f$ satisfies condition~\ref{thm:whale:b} of the theorem, that is, $f$ is the convolution of $m$ densities of exponential distributions and a completely monotone function $g$ on $(0, \infty)$ (which is integrable near $0$ and which converges to zero at infinity). Then, by a proposition proved in~\cite{simon} (see p.~889 therein), $f^{(n)}$ changes its sign $\min\{n,m\}$ times. Furthermore, it is easy to see that $f^{(n)}(0^+) = 0$ for $n = 0, 1, 2, \ldots, m - 1$, and hence $f$ is whale-shaped of order $m$.

Conversely, suppose that $f$ is a whale-shaped function on $(0, \infty)$ of order $m \ge 0$. We will prove that $f$ satisfies condition~\ref{thm:whale:b}. The argument follows closely the proof of Theorem~\ref{thm:bell2} in Section~\ref{sec:bs2}, and so we omit some details.

By an analogue of Lemma~\ref{lem:fn:lim}, for $n = 0, 1, 2, \ldots$,
\formula[eq:whale:inf]{
 & x^n f^{(n)}(x) \text{ converges to zero as } x \to \infty.
}
We claim that for $n = m, m + 1, m + 2, \ldots$
\formula[eq:whale:zero]{
 & x^{n - m} f^{(n)}(x) \text{ is integrable near zero and } \lim_{x \to 0^+} (x^{n - m + 1} f^{(n)}(x)) = 0 .
}
If $n = m = 0$, then integrability of $f^{(n)} = f$ near $0$ is one of the assumptions. If $n = m > 0$, $f^{(m)}$ is the derivative of $f^{(m - 1)}$, which, by assumption, is bounded and monotone in some right neighbourhood of zero. Thus, $f^{(n)} = f^{(m)}$ is integrable near $0$. We now proceed by induction. Suppose that for some $n > m$ we have already proved that $x^{n - m - 1} f^{(n - 1)}$ is integrable near $0$. By assumption, $f^{(n - 1)}$ and $f^{(n)}$ have constant signs in a right neighbourhood of $0$, and hence $f^{(n - 1)}$ has a constant sign and it is monotone near $0$. This property and integrability of $x^{n - m - 1} f^{(n - 1)}$ imply that $\lim_{x \to 0^+} (x^{n - m} f^{(n - 1)}(x)) = 0$. By integration by parts, the improper integral $\int_0^1 x^{n - m} f^{(n)}(x) dx$ exists. Since $f^{(n)}$ has constant sign in some right neighbourhood of $0$, this integral in fact converges absolutely. Our claim follows now by induction.

With~\eqref{eq:whale:inf} and~\eqref{eq:whale:zero} at hand, we proceed as in the proof of Theorem~\ref{thm:bell2}. Suppose that $n \ge m$, so that $f^{(n)}$ has $m$ zeroes in $(0, \infty)$, denoted by $n \alpha_{n,k}$, $k = 1, 2, \ldots, m$. As before, we let
\formula{
 g_n(x) & = \frac{(-1)^n n^{n + 1}}{n!} \, x^{n - m} f^{(n)}(n x) \prod_{k = 1}^m (x - \alpha_{n,k}) .
}
Note that this is the same definition as~\eqref{eq:gn2} if we agree that $\alpha_{n,k} = 0$ for $k > m$. Clearly, $g_n(x) \ge 0$ for all $x \in (0, \infty)$. For a fixed $z \in \C \setminus (-\infty, 0]$, we have
\formula{
 \frac{1}{1 + x z} \prod_{k = 1}^m (x - \alpha_{n,k}) & = \frac{(-1)^m}{z^m (1 + x z)} \prod_{k = 1}^m (1 + \alpha_{n,k} z) + P_n(x)
}
for all $x \in (0, \infty)$, where $P_n$ is a polynomial of degree at most $m - 1$. We claim that $n$-fold integration by parts leads to
\formula{
 \int_0^\infty x^{n - m} f^{(n)}(n x) P_n(x) dx & = 0 .
}
Indeed, let $h_{n,k}(x) = f^{(n - k)}(n x) (x^{n - m} P_n(x))^{(k - 1)}$ for $k = 1, 2, \ldots, n$. As $x \to \infty$, $h_{n,k}(x)$ is bounded by a constant times $x^{n - k} |f^{(n - k)}(n x)|$, and hence, by~\eqref{eq:whale:inf}, $h_{n,k}(x)$ converges to zero as $x \to \infty$. Similarly, for $k = 1, 2, \ldots, n - m$, the function $h_{n,k}(x)$ is bounded by a constant times $x^{n - m - k + 1} |f^{(n - k)}(n x)|$ as $x \to 0^+$, and hence, by~\eqref{eq:whale:zero}, $h_{n,k}(0^+) = 0$. Finally, if $k = n - m + 1, n - m + 2, \ldots, n$, then $h_{n,k}(x)$ is bounded by a constant times $|f^{(n - k)}(n x)|$ as $x \to 0^+$, and since $n - k < m$, we have $h_{n,k}(0^+) = 0$ by assumption. Our claim follows.

We have thus found that for $n = m, m + 1, m + 2, \ldots$ we have
\formula{
 \int_0^\infty \frac{g_n(x)}{1 + x z} \, dx & = \frac{(-1)^n n^{n + 1}}{n!} \int_0^\infty \frac{x^{n - m} f^{(n)}(n x)}{1 + x z} \biggl(\prod_{k = 1}^m (x - \alpha_{n,k})\biggr) dx \\
 & = \frac{(-1)^{n - m} n^{n + 1}}{n! \, z^m} \biggl(\prod_{k = 1}^m (1 + \alpha_{n,k} z)\biggr) \int_0^\infty \frac{x^{n - m} f^{(n)}(n x)}{1 + x z} \, dx .
}
Again we find that the above integrals are in fact absolutely convergent, and
\formula[eq:ws:aux]{
 \frac{(-1)^{n - m} n^{n + 1}}{n! \, z^m} \int_0^\infty \frac{x^{n - m} f^{(n)}(n x)}{1 + x z} \, dx & = \biggl(\prod_{k = 1}^m \frac{1}{1 + \alpha_{n,k} z}\biggr) \int_0^\infty \frac{g_n(x)}{1 + x z} \, dx .
}
Now we need an analogue of Lemma~\ref{lem:F:inv2}, with a different choice of the $n$-fold antiderivative of $1 / (s + i x)$ in formula~\eqref{eq:dpsi} in the proof: $(-i x / s)^{n - m} / (s + i x)$ rather than $1 / (s + i x)$. This variant applies to our case, with $z = i \xi$ and $\xi \in \R \setminus \{0\}$, and it implies that the left-hand side of~\eqref{eq:ws:aux} converges as $n \to \infty$ to $\laplace f(i \xi)$. We omit the details.

Above we have proved that for $\xi \in \R \setminus \{0\}$,
\formula{
 \laplace f(i \xi) & = \lim_{n \to \infty} \biggl(\prod_{k = 1}^m \frac{1}{1 + i \alpha_{n,k} \xi}\biggr) \int_0^\infty \frac{g_n(x)}{1 + i \xi x} \, dx .
}
This is a perfect analogue of~\eqref{eq:F2}. The remaining part of the proof is not much different from the corresponding argument in the proof of Theorem~\ref{thm:bell2}, and is therefore omitted.
\end{proof}

%
%                            ---------- o ----------
%

%
%                            ---------- o ----------
%


\begin{thebibliography}{00}

\bibitem{gawronski}
W.~Gawronski,
\emph{On the bell-shape of stable densities}.
Ann. Probab. 12(1) (1984): 230--242.

\bibitem{hsw}
T.~Hasebe, T.~Simon, M.~Wang,
\emph{Some properties of the free stable distributions}.
Ann. Inst. Henri Poincaré, in press.

\bibitem{hirschman}
I.~I.~Hirschman,
\emph{Proof of a conjecture of I.~J.~Schoenberg}.
Proc. Amer. Math. Soc.~1 (1950): 63--65.

\bibitem{js}
W.~Jedidi, T.~Simon,
\emph{Diffusion hitting times and the bell-shape}.
Stat. Probab. Lett. 102 (2015): 38--41.

\bibitem{karlin}
S.~Karlin,
\emph{Total positivity. Vol. 1}.
Stanford University Press, Stanford, CA, 1968.

\bibitem{bell}
M.~Kwaśnicki,
\emph{A new class of bell-shaped functions}.
Trans. Amer. Math. Soc., in press, arXiv:1710.11023.

\bibitem{rogers}
M.~Kwa\'snicki,
\emph{Fluctuation theory for L\'evy processes with completely monotone jumps}.
Electron. J. Probab. 24 (2019), no.~40: 1--40.

\bibitem{ssv}
R.~Schilling, R.~Song, Z.~Vondra{\v{c}}ek,
\emph{Bernstein Functions: Theory and Applications}.
Studies in Math. 37, De Gruyter, Berlin, 2012.

\bibitem{schoenberg1}
I.~J.~Schoenberg,
\emph{On Totally Positive Functions, Laplace Integrals and Entire Functions of the Laguerre--Polya--Schur Type}.
Proc. Nat. Acad. Sci. 33(1) (1947): 11--17.

\bibitem{schoenberg2}
I.~J.~Schoenberg,
\emph{On Variation-Diminishing Integral Operators of the Convolution Type}.
Proc. Nat. Acad. Sci. 34(4) (1948): 164--169.

\bibitem{simon}
T.~Simon,
\emph{Positive stable densities and the bell-shape}.
Proc. Amer. Math. Soc. 143(2) (2015): 885--895.

\bibitem{hw}
D.~V.~Widder, I.~I.~Hirschman,
\emph{The Convolution Transform}.
Princeton Math. Ser. 20, Princeton University Press, Princeton, 1955.

\end{thebibliography}
\end{document}